\documentclass[british]{amsart}
\usepackage[T1]{fontenc}
\usepackage[latin9]{inputenc}
\usepackage[a4paper]{geometry}
\usepackage{amstext}
\usepackage{amsthm}
\usepackage{amssymb}
\usepackage{stmaryrd}
\usepackage{setspace}
\makeatletter
\numberwithin{equation}{section}
\numberwithin{figure}{section}
\theoremstyle{plain}
\newtheorem{thm}{\protect\theoremname}[section]
  \theoremstyle{definition}
  \newtheorem{defn}[thm]{\protect\definitionname}
  \theoremstyle{plain}
  \newtheorem{cor}[thm]{\protect\corollaryname}
  \theoremstyle{plain}
  \newtheorem{prop}[thm]{\protect\propositionname}
  \theoremstyle{plain}
  \newtheorem{lem}[thm]{\protect\lemmaname}
  \theoremstyle{plain}
  \newtheorem{fact}[thm]{\protect\factname}
  \theoremstyle{remark}
  \newtheorem*{rem*}{\protect\remarkname}
  \theoremstyle{remark}
  \newtheorem{rem}[thm]{\protect\remarkname}

\usepackage{amscd}
\usepackage{mathptmx}
\usepackage{bbm}
\usepackage{paralist}

\newcommand{\e}{\mathrm{e}}
\newcommand{\1}{1}
\newcommand{\N}{\mathbb{N}}

\newcommand{\R}{\mathbb{R}}

\renewcommand{\Pi}{\pi}
\renewcommand{\emptyset}{\varnothing}

\newcommand\diam{\mathrm{diam}}

\DeclareMathOperator*{\Int}{Int}

\newcommand\T{\mathrm{\mathcal{T}}}
\DeclareMathOperator*{\Span}{span}

\DeclareMathOperator*{\card}{card}

\DeclareMathOperator*{\Lker}{J_{ker}}

\renewcommand{\vec}[1]{\mathbf{#1}}
\DeclareMathOperator*{\Fix}{Fix}

\DeclareMathOperator*{\Lip}{Lip}

\DeclareMathOperator*{\Hol}{\textnormal{H\"ol}}

\usepackage{hyperref}
\@ifundefined{textmu}
 {\usepackage{textcomp}}{}

\makeatother

\usepackage{babel}
  \providecommand{\corollaryname}{Corollary}
  \providecommand{\definitionname}{Definition}
  \providecommand{\factname}{Fact}
  \providecommand{\lemmaname}{Lemma}
  \providecommand{\propositionname}{Proposition}
  \providecommand{\remarkname}{Remark}
\providecommand{\theoremname}{Theorem}

\begin{document}

\title[Spectral gap property for random dynamics on the real line]{Spectral gap property for random dynamics on the real line  and
multifractal analysis of generalised Takagi functions }
\begin{abstract}
We consider the random iteration of finitely many expanding $\mathcal{C}^{1+\epsilon}$
diffeomorphisms on the real line without a common fixed point. We
derive the spectral gap property of the associated transition operator acting
on spaces of H\"older continuous functions. As an application we introduce 
generalised Takagi
functions on the real line and we perform a complete multifractal analysis of the pointwise H\"older exponents of these functions.
\end{abstract}

\keywords{Random dynamical systems, iterated function systems, fractal geometry, fractal functions, multifractal analysis, thermodynamic formalism.}

\author{Johannes Jaerisch}

\address{Graduate School of Mathematics, Nagoya University, Furocho, Chikusaku, 
Nagoya, 464-8602 Japan}

\email{jaerisch@math.nagoya-u.ac.jp}

\urladdr{ 	
http://www.math.nagoya-u.ac.jp/\textasciitilde{}jaerisch/index.html}

\author{Hiroki Sumi}

\address{Course of Mathematical Science, Department of Human Coexistence,
Graduate School of Human and Environmental Studies, Kyoto University
Yoshida-nihonmatsu-cho, Sakyo-ku, Kyoto 606-8501, Japan }

\email{sumi@math.h.kyoto-u.ac.jp}

\urladdr{http://www.math.h.kyoto-u.ac.jp/\textasciitilde{}sumi/index.html}

\date{\today, \ {\em MSC2010}: 37H10, 37D35 (primary), 28A80 (secondary). To appear 
	in Comm. Math. Phys. }

\maketitle

\section{Introduction and statement of results}
In this paper, we investigate the independent and identically-distributed (i.i.d.) random dynamical systems on the real line. The theory of dynamical systems is used to describe various subjects in basically all areas of natural and social sciences. Since nature and any other environment have a lot of random terms, it is very natural and important not only to consider the dynamics of iteration of one map, but also to consider random dynamics. Many researchers in various fields have found and investigated many kinds of new phenomena in random dynamics which cannot hold in deterministic dynamics. These phenomena arise from the effect of randomness or noise and they are called {\emph{ randomness-induced phenomena}} or {\emph{ noise-induced phenomena}} (\cite{JS13b, JS2016, s11random, S13Coop}). Under certain conditions, because of 
the effect of randomness or noise, the chaoticity of the system becomes milder, but the system  still has some complexity. Hence regarding such random dynamical systems, our aim is to investigate the
{\emph{ gradation between  chaos and order}}.  

To find and to study quantities describing the gradation between chaos and order, we combine ideas  of random dynamical systems, ergodic theory (in particular, thermodynamic formalism), iterated function systems,  and fractal geometry.  More precisely, for any   random dynamical system in our setting,  there exists 
 an exponent $\alpha _{-}\in (0,1)$ such that for each 
 $\alpha $ with $0<\alpha <\alpha _{-}$, 
 the transition operator of the system behaves 
 well (e.g.,  it has a spectral gap property) on 
 the space ${\mathcal C}^{\alpha }$ of $\alpha$-H\"older continuous functions 
 endowed with $\alpha$-H\"older norm, 
 but for each $\alpha $ with $\alpha _{-}<\alpha <1$, 
 the transition operator of the system does not behave 
 well (Theorem~\ref{thm:specgapintro}, Corollary~\ref{cor:transitionintro}). This quantity $\alpha _{-}$ 
 describes the  gradation between chaos and order  
for the system.  Furthermore, to provide a refined  gradation, we investigate the pointwise 
H\"older exponents of the limit state functions (i.e., fixed points of the transition operator) and 
their (higher order) partial derivatives with respect 
to the probability parameters. It turns out that the pointwise 
H\"older exponents have a complicated fine structure which can be suitably investigated using 
the multifractal analysis and the concept of fractal dimension  (Theorems~\ref{thm:dimspecintro}, \ref{intro thm:global hoelder}).   The objects appearing in the multifractal analysis also describe the gradation between  chaos and order for the system. Moreover, we present a new general framework to study a large class of fractal functions. In particular, we shed new light on the regularity properties of the classical Takagi function in our framework (see Theorem \ref{intro thm:global hoelder}  and Proposition \ref{intro prop:non-differentiable}).

Throughout, let $I:=\left\{ 1,\dots,s+1\right\} $, $s\ge1$, and
let $f_{i}:\R\rightarrow\R$, $i\in I$, be a family of $\mathcal{C}^{1+\epsilon}$
diffeomorphisms with $\epsilon$-H\"older continuous derivatives
for some  $\epsilon>0$. We say that $(f_{i})_{i\in I}$ is \emph{ expanding}
if there exists $\lambda>1$ such that $f_{i}'(x)\ge\lambda>1$, for
all $x\in\R$ and $i\in I$. The family $(f_{i})_{i\in I}$ \emph{has
no common fixed point} if there exists no $x\in\R$ such that $f_{i}(x)=x$
for all $i\in I$. 

We denote by $\overline{\R}:=\R\cup\left\{ \pm\infty\right\} $ the
two-point compactification of $\R$ endowed with a metric $d$ on
$\overline{\R}$ which is strongly equivalent to the Euclidean metric
on compact subsets of $\R$, that is, for each compact set $K\subset\R$
there exists a constant $C>0$ such that $C^{-1}|x-y|\le d(x,y)\le C|x-y|$,
for all $x,y\in K$. For $i\in I$ we extend the definition of $f_{i}$
from $\R$ to $\overline{\R}$ by setting $f_{i}(\pm\infty):=\pm\infty.$
We say that $(f_{i})_{i\in I}$ is \emph{contracting near infinity}
if there exist neighborhoods $V^{\pm}$ of $\pm\infty$ such that
$\Lip(f_{i|V^{\pm}})<1$, $i\in I$. Here, for $D\subset\overline{\R}$
and $g:D\rightarrow\R$, we have set $\Lip(g):=\sup_{x,y\in D,x\neq y}d\left(g(x),g(y)\right)\big/d(x,y)$.
Note that if $(f_{i})_{i\in I}$ is contracting near infinity then
$\Lip(f_{i})<\infty$ for each $i\in I$. We refer to Section 
\ref{sec:Contractions-near-infinity} for details about the property
of contraction near infinity. 

Throughout, we assume that $(f_{i})_{i\in I}$ is  expanding, has
no common fixed point, and is contracting near infinity. For $\vec{p}=(p_{1},\dots,p_{s})\in\left(0,1\right)^{s}$
with $\sum_{i=1}^{s}p_{i}<1$, let $p_{s+1}:=1-\sum_{i=1}^{s}p_{i}$.
Let $\mathcal{C}(\overline{\R})$ denote the Banach space of continuous 
functions endowed
with the supremum norm $\Vert\cdot\Vert_{\infty}$. Define the transition
operator 
\[
M_{\vec{p}}:\mathcal{C}\left(\overline{\R}\right)\rightarrow\mathcal{C}\left(\overline{\R}\right),\quad M_{\vec{p}}h=\sum_{i\in I}p_{i}\cdot h\circ f_{i},\,\,h\in\mathcal{C}\left(\overline{\R}\right).
\]
For $\alpha>0$ let $\mathcal{C}^{\alpha}(\overline{\R})$ denote
the Banach space of $\alpha$-H\"older continuous functions (see
Section \ref{sec:Spectral-gap-property}). Note that $M_{\vec{p}}\left(\mathcal{C}^{\alpha}(\overline{\R})\right)\subset\mathcal{C}^{\alpha}(\overline{\R})$. To state
our first main result we say that $M_{\vec{p}}$ has the \emph{spectral
gap property} if its spectrum consists of finitely many eigenvalues
of modulus one, and the rest of the spectrum is contained in a ball
of radius strictly less than one. We say that $(f_{i})_{i\in I}$
satisfies the \emph{separating condition}  if there exists a non-empty
bounded open interval $O\subset\R$ such that $f_{i}^{-1}(O)\subset O$,
for all $i\in I$, and for all $i,j\in I$ with $i\neq j$, we have
$f_{i}^{-1}(\overline{O})\cap f_{j}^{-1}(\overline{O})=\emptyset.$
 For the definition of the bottom of the spectrum $\alpha_{-}=\alpha_{-}(\vec{p})$
we refer to Section \ref{subsec:Improved-spectral-gap}. For $\vec{a}\in\R^{s}$
and $\delta>0$ we denote by  $B(\vec{a},\delta)\subset\R^{s}$ 
the open ball of radius $\delta$ with center $\vec{a}$ in $\R^{s}$. 
\begin{thm}[Theorem \ref{thm:spectralgap} and Theorem \ref{thm: spectral gap separating condition}]
\label{thm:specgapintro}
For every $\vec{p}_{0}\in(0,1)^{s}$ there exist $\delta>0$ and
$\alpha>0$ such that $M_{\vec{p}}:\mathcal{C}^{\alpha}(\overline{\R})\rightarrow\mathcal{C}^{\alpha}(\overline{\R})$
has the spectral gap property for every $\vec{p}\in B(\vec{p}_{0},\delta)$.
If $(f_{i})_{i\in I}$ satisfies the separating condition, then the
previous assertion holds for any $\alpha<\alpha_{-}(\vec{p}_{0})$. 
\end{thm}

By combining with the perturbation theory of linear operators we
can derive that the probability of tending to infinity $T:=T_{\vec{p}}:\overline{\R}\rightarrow\left[0,1\right]$
(see \eqref{eq:def prob infinty} below for the definition) depends real analytically on $\vec{p}$
(see Theorem \ref{thm:spectralgap} for the detailed statement). This
allows us to make the following definition. Let $\N_0:=\{0,1,\dots \}$.
\begin{defn}
\label{intro def:generalised Takagi}We denote by $\T:=\T_{\vec{p}}$
the $\R$-vector space of generalised Takagi functions generated by
\[
C_{\vec{n}}(x):=C_{\vec{n},\vec{p}}(x):=\frac{\partial^{\sum_{i=1}^{s}n_{i}}}{\partial u_{1}^{n_{1}}\partial u_{2}^{n_{2}}\dots\partial u_{s}^{n_{s}}}T_{(u_{1},\dots,u_{s})}(x)\big|_{(u_{1},\dots,u_{s})=\vec{p},\quad}\vec{n}=(n_{1},\dots,n_{s})\in\N_{0}^{s},\,\,x\in\overline{\R}.
\]
\end{defn}
We say that an element
$C\in\T$ is non-trivial if there exists $(\beta_{\vec{n}})_{\vec{n}}\neq0$
such that $C=\sum_{\vec{n}}\beta_{\vec{n}}C_{\vec{n}}$. For the reason why we call the elements of  $\T$ generalised Takagi functions, we refer to Remark \ref{rem:sharp alpha minus}. 
We then proceed to investigate the regularity of the  elements of $\T$.
The\emph{ }pointwise H\"older exponent of $C\in\T$ at $x\in\R$
is denoted by $\Hol(C,x)$ (see (\ref{eq:hoelder exponent as liminf}) below for the definition).
We denote by $J$ the Julia set of $(f_{i})_{i\in I}$ (see Section
\ref{subsec:General-results}). For $\vec{p}\in(0,1)^{s}$ we define
$\alpha_{+}=\alpha_{+}(\vec{p})$ in Section \ref{subsec:Improved-spectral-gap}.
We say that $(f_{i})_{i\in I}$ satisfies the open set condition if
there exists a non-empty bounded open interval $O\subset\R$ such
that $f_{i}^{-1}(O)\subset O$, for all $i\in I$, and for all $i,j\in I$
with $i\neq j$ we have $f_{i}^{-1}(O)\cap f_{j}^{-1}(O)=\emptyset$.
  By $t^{*}$ we denote the Legendre transform of the function $t$
defined implicitly by a certain topological pressure functional (see
Section \ref{subsec:Dimension-spectrum-of}).  Denote by $\dim_H(A)$ the 
Hausdorff dimension of a set $A\subset \R$ with respect to the Euclidean metric.

\begin{thm}[Theorem \ref{thm:multifrac}]
\label{thm:dimspecintro}
Suppose that $(f_{i})_{i\in I}$ satisfies the open set condition.
Let $C\in\T$ be non-trivial. Then we have for all $\alpha\in[\alpha_{-},\alpha_{+}],$
\[
\dim_{H}\left\{ x\in J\mid\Hol(C,x)=\alpha\right\} =-t^{*}(-\alpha),
\]
and for $\alpha\not\notin[\alpha_{-},\alpha_{+}]$ we have $\left\{ x\in 
J\mid\Hol(C,x)=\alpha\right\} =\emptyset$. The function $g(\alpha):= 
-t^{*}(-\alpha)$ is continuous and concave on 
$[\alpha_{-},\alpha_{+}]$. If $\alpha_{-}<\alpha_{+}$ then  $g$ is 
real-analytic and positive on $(\alpha_{-},\alpha_{+})$ and satisfies $g''<0$ 
on  $(\alpha_{-},\alpha_{+})$.
\end{thm}

We prove the following result regarding the global H\"older continuity of elements of $\T$.
\begin{thm}[Theorem \ref{thm:alpha-minus hoelder}, Corollary \ref{cor:optimal hoelder continuity}]
\label{intro thm:global hoelder}Suppose that $(f_{i})_{i\in I}$
satisfies the open set condition. Let $C\in\T$ be non-trivial. Then
we have $\alpha_{-}=\sup\left\{ \alpha\ge0\mid C\in\mathcal{C}^{\alpha}(\overline{\R})\right\} $.
Further, we have $T\in\mathcal{C}^{\alpha_{-}}\left(\overline{\R}\right)$. 
\end{thm}
The
following corollary indicates that the random dynamical system generated by 
$(f_{i})_{i\in
	I}, (p_{i})_{i\in I}$ still has some kind of 
complexity. 
\begin{cor}[Corollary \ref{cor:transition operator chaotic on some Calpha}]
\label{cor:transitionintro}
Suppose that $(f_{i})_{i\in I}$ satisfies the open set condition.
If $\alpha_{-}<1$ then \newline $\lim_{n\rightarrow\infty}\Vert M_{\vec{p}}^{n}\Vert_{\alpha}=\infty$,
for each $\alpha_{-}<\alpha<1$, where $\Vert M_{\vec{p}}^{n}\Vert_{\alpha}$
denotes the operator norm of $M_{\vec{p}}^{n}$ on $\mathcal{C}^{\alpha}(\overline{\R})$. 
\end{cor}

Regarding the existence of points of non-differentiability of elements
of $\T$ we prove the following. Let $\vec{e}_{k}\in\N^{s}_0$ denote
the $k$-th unit vector in $\N_{0}^{s}$, $1\le k\le s$. We use $C_{m}$
to denote $C_{(m)}$ for $m\in\N_{0}$. 
\begin{prop}[Proposition \ref{prop:non-differentiability}]
\label{intro prop:non-differentiable} Suppose that $(f_{i})_{i\in I}$
satisfies the open set condition. 
\begin{enumerate}
\item If $\alpha_{-}<1$,  then there exists a dense subset $E\subset J$
of positive Hausdorff dimension such that, for every non-trivial $C\in\T$
and every $x\in E$, $C$ is not differentiable at $x$.
\item If $\alpha_{-}=1$,  $s=1$ and $f_1'$ and $f_2'$ are constant functions,  
then $C_{m}$ is
nowhere differentiable on $J$, for every $m\ge1$. 
\end{enumerate}
\end{prop}

For applications of our results to conjugacies of interval maps we
refer to Section \ref{sec:Conjugaries-between-Interval}. In fact,
if $(f_{i})_{i\in I}$ satisfies the open set condition, then the
probability of tending to infinity $T_{\vec{p}}$ can also be characterised
as the conjugacy map between the expanding dynamical system defined
by $(f_{i})_{i\in I}$ on $J$ and the dynamical system given by the
piecewise linear map on $[0,1]$ with $(s+1)$ full branches and slopes
given by $(1/p_{i})_{i\in I}$ (see Lemma \ref{lem:T is conjugacy}).
By the rigidity dichotomy in \cite[Theorem 1.2]{MR2576266}, if  $J$
is an interval, then either $\alpha_{-}(\vec{p})=\alpha_{+}(\vec{p})$
and $T_{\vec{p}}$ is a $\mathcal{C}^{1+\epsilon}$-diffeomorphism,
or $\alpha_{-}(\vec{p})<\alpha_{+}(\vec{p})$ and the set of non-differentiability
points of $T_{\vec{p}}$ has positive Hausdorff dimension. For general
families $(f_{i})_{i\in I}$ satisfying the open set condition, we
can show that $\alpha_{-}(\vec{p})=\alpha_{+}(\vec{p})$ if and only
if $T_{\vec{p}}\in\mathcal{C}^{\dim_{H}(J)}(\overline{\R})$ (see
Section \ref{sec:Conjugaries-between-Interval}). 

Higher order derivatives of the classical Takagi function have been
considered in \cite{MR2212280}, where it is shown that the classical
Takagi function and the higher order derivatives of the Lebesgue singular
function for $p=1/2$ are nowhere differentiable and convex Lipschitz
(\cite{MR860394}). These results are covered by our general theory.
Namely, since for this special case we have $\alpha_{-}=1$ and $s=1$,
the non-differentiability follows from Proposition \ref{intro prop:non-differentiable}
(2). That these functions are $\alpha$-H\"older continuous, for
every $\alpha<1$, follows from Theorem \ref{intro thm:global hoelder}.
In fact, we can also derive that the functions are convex Lipschitz
(see Remark \ref{rem:convex lipschitz}). 

Generalised Takagi functions (with respect to the
parameter $\vec{p}$) have also been introduced in \cite{MR839313}.
As already observed in \cite{MR2212280}, the higher-order derivatives
of   $T_{\vec{p}}$ of a system  $(f_{i})_{i\in I}$ with constant derivatives are not
covered by the setting in \cite{MR839313}. In \cite{MR1111613} it
is shown that the Lebesgue singular function depends real analytically
on the parameter, and its higher order derivatives are considered.
We point out that our general definition of $\T$ is a far-reaching
generalisation of the above concept, where we consider an arbitrary finite number of $\mathcal{C}^{1+\epsilon}$
diffeomorphisms and arbitrary linear combinations of higher order
partial derivatives of the probability of tending to infinity with
$s\ge1$ parameters.

We remark that in the previous works of the authors \cite{JS13b, JS2016} we dealt with the random complex dynamical systems satisfying the separating condition. However, in this paper, we deal with random dynamical systems on the real line satisfying the open set condition. Note that the separating condition implies the open set condition. When we deal with general systems satisfying the open set condition,  we have to overcome new difficulties. In fact, the relation between the pointwise H\"older exponents of elements of $\T$  and the corresponding dynamical quantities is much more involved than in the case of the separating condition. We developed several new ideas (see Propositions \ref{prop:lower bound} and \ref{prop:lower half of spectrum upper bound}) to overcome these difficulties. In the case of the separating condition we show that $\alpha_-$ is the supremum of the exponents $\alpha$ for which the transition operator has the spectral gap property on $\mathcal{C}^{\alpha}$ by developing some idea from \cite{JS2016} and providing a new approach.

In Section~\ref{sec:Spectral-gap-property}, we derive the spectral gap property 
for the 
transition operator associated with random dynamical systems on the real line. 
In Section \ref{section:Multifractal}, we perform a complete multifractal 
analysis of the
pointwise H\"older exponents of the elements of $\T$ associated
with $(f_{i})_{i\in I}$ and $\vec{p}\in(0,1)^{s}$. 
In Section~\ref{sec:Global-Hlder-continuity},
 we investigate the global H\"older continuity of
the elements of $\T$. In Section~\ref{sec:Non-differentiability}, 
we study  the (non-)differentiability of the
elements of $\T$. In Section~\ref{sec:Conjugaries-between-Interval}, 
we show how our results are related to interval conjugacy
maps. In Section~\ref{sec:Contractions-near-infinity} (Appendix), 
we will show that, by modifying
the $(f_{i})_{i\in I}$ near infinity, we can always assume that an 
 expanding family $(f_{i})_{i\in I}$ is contracting near infinity
with respect to a metric $d$ which is strongly equivalent to the
Euclidean metric on compact subsets of $\R$.  

{\bf Acknowledgements.}  
The authors would like to thank Rich Stankewitz for valuable comments. The research of the first author was partially supported by JSPS Kakenhi 15H06416, 17K14203. The research of the second author was partially supported by JSPS Kakenhi 15K04899, 18H03671. 

\section{Spectral gap property\label{sec:Spectral-gap-property}}

In this section we derive the spectral gap property for the transition
operator associated with random dynamical systems on the real line
induced by a family $(f_{i})_{i\in I}$. 

\subsection{General results \label{subsec:General-results}}

Let $I^{*}:=\bigcup_{n\in\N}I^{n}$. For $\omega\in I^{*}$ we denote
by $\left|\omega\right|$ the unique $n\in\N$ such that $\omega\in I^{n}$.
For $\omega=\left(\omega_{1},\dots,\omega_{n}\right)\in I^{n}$ we
let $f_{(\omega_{1},\dots,\omega_{n})}:=f_{\omega_{n}}\circ\dots\circ f_{\omega_{1}}$.
Let $\Sigma := I^{\Bbb{N}}$. Also, for $\omega\in \Sigma$ and $n\in\N$ we put 
$\omega_{|n}:=(\omega_{1},\dots,\omega_{n})\in I^{n}$.
Since $(f_{i})_{i\in I}$ is contracting near infinity, there exist
neighborhoods $V^{\pm}$ of $\pm\infty$ in $\overline{\R}$ such
that, for each $x\in V^{+}$ (resp.\,\,  $x\in V^{-}$) we have for all
$\omega\in \Sigma$, 
\[
f_{\omega_{|n}}(x)\rightarrow+\infty\quad(\text{resp.}f_{\omega_{|n}}(x)\rightarrow-\infty),\quad\text{as }n\rightarrow\infty.
\]
We put 
\[
V:=V^{+}\cup V^{-}.
\]
Denote by $G:=\left\langle f_{1},\dots f_{s+1}\right\rangle :=\left\{ f_{\omega}\mid\omega\in I^{*}\right\} $
the semigroup generated by $f_{1},\dots,f_{s+1}$ where the semigroup
operation is the composition of functions. The\emph{ Julia set of
$G$ }is defined as 
\[
J:=\left\{ x\in\overline{\R}\mid G\text{ is not equicontinuous in any neighborhood of }x \text{ with respect to }d\right\} .
\]
Note that the inverse maps $(f_{i}^{-1})_{\big|\overline{\R}\setminus V},i\in I$,
form a contracting conformal iterated function system (see e.g. \cite{falconerfractalgeometryMR2118797,MR2003772})
and $J$ is the limit set (or attractor) of this system. 

For $\omega=\left(\omega_{1},\omega_{2},\dots\right)\in\Sigma$
we define 
\[
J_{\omega}:=\bigcap_{n\in\N}\left(f_{\omega_{|n}}\right)^{-1}\left(\overline{\R}\setminus V\right).
\]
Note that $J_{\omega}$ is a singleton because $(f_{i})_{i\in I}$
is  expanding. We define the coding map $\pi:\Sigma\rightarrow\R$
given by 
\[
\bigcap_{n\in\N}(f_{\omega_{|n}})^{-1}(\overline{\R}\setminus V)=\left\{ \pi(\omega)\right\} ,\quad\omega\in\Sigma.
\]
It is easy to see that 
\[
J=\bigcup_{\omega\in\Sigma}J_{\omega}=\pi(\Sigma).
\]

The \emph{kernel Julia set of $G$} (\cite{s11random}) is given by
\[
\Lker:=\bigcap_{g\in G}g^{-1}\left(J\right)\subset J.
\]

For $\vec{p}=(p_{1},\dots,p_{s})\in\left(0,1\right)^{s}$ with $\sum_{i=1}^{s}p_{i}<1$,
let $p_{s+1}:=1-\sum_{i=1}^{s}p_{i}$. Let $\mu_{\vec{p}}$ denote
the $\left(p_{1},\dots,p_{s},p_{s+1}\right)$ Bernoulli measure on
$\Sigma$. 
\begin{prop}
\label{prop:almost everywhere converges to minimal sets}We have $\Lker=\emptyset$.
Moreover, we have $\mu_{\vec{p}}\left(\left\{ \omega\in\Sigma\mid x\in J_{\omega}\right\} \right)=0$,
for each $x\in\R$. 
\end{prop}

\begin{proof}
Since $(f_{i})_{i\in I}$ is expanding without a common fixed point,
for all $x\in\R$ there exists $g_{x}\in G$ such that
$g_{x}(x)\in V$. Hence, $\Lker=\emptyset$. By \cite[Lemma 4.6]{s11random}
we have $\mu_{\vec{p}}\left(\left\{ \omega\in\Sigma\mid x\in J_{\omega}\right\} \right)=0$,
for each $x\in\R$. 
\end{proof}
Recall that $M_{\vec{p}}$ is \emph{almost periodic} if $(M_{\vec{p}}^{n}h)_{n\ge1}$
is relatively compact in $\mathcal{C}(\overline{\R})$, for each $h\in\mathcal{C}(\overline{\R})$.
The dual operator of $M_{\vec{p}}$ is given by $M_{\vec{p}}^{*}:\mathcal{M}_{1}(\overline{\R})\rightarrow\mathcal{M}_{1}(\overline{\R})$,
where $\mathcal{M}_{1}(\overline{\R})$ denotes the space of Borel
probability measures on $\overline{\R}$ endowed with the topology
of weak convergence. We define the compact subset 
\[
J_{\text{meas}}:=\left\{ m\in\mathcal{M}_{1}(\overline{\R})\mid(M_{\vec{p}}^{*})_{n\ge1}^{n}\text{ is not equicontinuous in any neighbourhood of }m\right\} \subset\mathcal{M}_{1}(\overline{\R}).
\]
The following fact is a special case of \cite[Proposition 4.7, Lemma 4.2(6)]{s11random}. 
\begin{prop}
We have that $J_{\text{meas}}=\emptyset$ and that $M_{\vec{p}}:\mathcal{C}(\overline{\R})\rightarrow\mathcal{C}(\overline{\R})$
is almost periodic. 
\end{prop}

By a well-known result of Ljubich (\cite{MR741393}) on almost periodic
operators, we have 
\[
\mathcal{C}(\overline{\R})=\left\{ h\in\mathcal{C}(\overline{\R})\mid\Vert M_{\vec{p}}^{n}h\Vert_{\infty}\rightarrow0,\,\,\,\text{as }n\rightarrow\infty\right\} \varoplus\overline{\Span\left\{ h\in\mathcal{C}(\overline{\R})\mid\exists\rho\in\mathbb{S}^{1}\mid M_{\vec{p}}h=\rho h\right\} .}
\]

As in \cite{s11random} we define the probability of tending to infinity
\begin{equation}
T_{\vec{p}}:\overline{\R}\rightarrow\left[0,1\right],\quad T_{\vec{p}}(x):=\mu_{\vec{p}}\left\{ \omega\in\Sigma\mid\lim_{n\rightarrow\infty}f_{\omega_{|n}}(x)=\infty\right\} .\label{eq:def prob infinty}
\end{equation}
 It follows from Proposition \ref{prop:almost everywhere converges to minimal 
 sets} and the dominated convergence theorem that for every $h\in 
 \mathcal{C}(\overline{\R})$ that 
\begin{equation}
\lim_{n\rightarrow\infty}M_{\vec{p}}^{n}h(x)=\lim_{n\rightarrow\infty} \int 
h\circ f_{\omega_n}\circ 
\dots \circ 
f_{\omega_1}(x)\,\,d\mu_{\vec{p}}(\omega)=T_{\vec{p}}(x)h(\infty)+
(1-T_{\vec{p}}(x))h(-\infty).\label{eq:Mp
 convergence measurable}
\end{equation}
Hence,
\[
\Span\left\{ h\in\mathcal{C}(\overline{\R})\mid\exists\rho\in\mathbb{S}^{1}\mid M_{\vec{p}}h=\rho h\right\} =\R T_{\vec{p}}\varoplus\R1
\]
and we have by Ljubich's result, 
\[
\Vert M_{\vec{p}}^{n}h-(h(\infty)-h(-\infty))T_{\vec{p}}-h(-\infty)1\Vert_{\infty}\rightarrow0,\quad\text{as }n\rightarrow\infty.
\]

For $\alpha>0$ we say that $h:\overline{\R}\rightarrow\R$ is $\alpha$-H\"older
continuous if 
\[
V_{\alpha}:=\sup_{x\neq y}\left\{ \frac{\left|h(x)-h(y)\right|}{d(x,y)^{\alpha}}\right\} <\infty.
\]
We say that a function is H\"older continuous if it is $\alpha$-H\"older
continuous for some $\alpha>0$. We denote by $\mathcal{C}^{\alpha}(\overline{\R})$
the Banach space of $\alpha$-H\"older continuous maps on $\overline{\R}$
endowed with the $\alpha$-H\"older norm 
\[
\Vert h\Vert_{\alpha}:=V_{\alpha}+\Vert h\Vert_{\infty},\quad h\in\mathcal{C}^{\alpha}(\overline{\R}).
\]
The following lemma is the key to derive the spectral gap property
of $M_{\vec{p}}$ on $\mathcal{C}^{\alpha}(\overline{\R})$. The proof
is inspired by \cite{S13Coop}. For $n\in\N$ and $\omega\in I^{n}$
we will use the notation 
\[
p_{\omega}:=p_{\omega_{1}}\cdot\dots\cdot p_{\omega_{n}}.
\]
Hence, we have for $n\in\N$ and $h\in\mathcal{C}\left(\overline{\R}\right)$
\[
M_{\vec{p}}^{n}h=\sum_{\omega\in I^{n}}p_{\omega}\cdot(h\circ f_{\omega})=\int h\circ f_{\omega_{|n}}\,\,d\mu_{\vec{p}}(\omega).
\]
\begin{lem}
\label{lem:doeblin-fortetinequality general}For every $\vec{p}_{0}\in(0,1)^{s}$
there exist $\delta>0$, $\alpha>0$, $n\in\N$ and constants $0<c<1$
and $C>0$ such that, for every $\vec{p}\in B(\vec{p}_{0},\delta)$
and for every $h\in\mathcal{C}^{\alpha}(\overline{\R})$, 
\[
\left|M_{\vec{p}}^{n}h(x)-M_{\vec{p}}^{n}h(y)\right|\le\left(c\Vert h\Vert_{\alpha}+C\Vert h\Vert_{\infty}\right)d(x,y)^{\alpha},\quad x,y\in\overline{\R}.
\]
\end{lem}

\begin{proof}
Recall that $V=V^{+}\cup V^{-}$. Since $(f_{i})_{i\in I}$ is expanding
without a common fixed point, we have that, for all $x\in\R$ there
exists $g_{x}\in G$ and a compact neighborhood $U_{x}$ of $x$ in
$\overline{\R}$ such that $g_{x}(U_{x})\subset V$. Since $\overline{\R}$
is compact, there exist $t\in\N$ and $x_{1},\dots,x_{t}\in\overline{\R}$
such that $\overline{\R}=\bigcup_{j=1}^{t}\Int(U_{x_{j}})$, where
$\Int(A)$ denotes the interior of a set $A\subset\overline{\R}$.
Since $\bigcup_{g\in G}g(V)\subset V$ we may assume that there exists
$r\in\N$ such that, for each $j=1,\dots,t$ there exists $\beta^{j}\in I^{r}$
with $g_{x_{j}}=f_{\beta^{j}}$. For $\omega\in I^{*}$ with $|\omega|=\ell$
we denote by $\left[\omega\right]:=\left\{ \tau\in\Sigma\mid\tau_{1}=\omega_{1},\dots,\tau_{\ell}=\omega_{\ell}\right\} $
the cylinder set of $\omega$. Let 
\[
a:=a(\vec{p}):=\max\left\{ 1-\mu_{\vec{p}}([\beta^{j}])\mid1\le j\le t\right\} <1.
\]
Recall that $\Lip(g)<\infty$, for every $g\in G$. Let 
\[
\Lambda:=2\cdot\max\left\{ \max\left\{ \Lip\left(f_{\omega}\right)\mid\omega\in 
I^{r}\right\} ,1\right\} \ge2.
\]
Let $R>0$ be a Lebesgue number of the covering $\left(\Int(U_{x_{j}})\right)_{1\le j\le t}$
of $\overline{\R}.$ Let $\alpha>0$ such that 
\[
\eta:=a\Lambda^{\alpha}<1.
\]
Let $n\in\N$ to be determined later. Let $x,y\in\overline{\R}$.
Since $M_{\vec{p}}^n$ has norm one, we may assume that $d(x,y)<R$. Let
\[
n(x,y):=\max\left\{ k\ge0\mid\Lambda^{-k}R>d(x,y)\right\} .
\]
If $n(x,y)<n$, then $d(x,y)\ge\Lambda^{-n}R$ and the desired estimate
follows with $C:=2\Lambda^{n\alpha}R^{-\alpha}$. Now, we consider
the case $n(x,y)\ge n$. Then we have $d(x,y)<\Lambda^{-n}R$. Consequently,
for $j\le n$ and $\omega\in I^{jr}$ we have $d(f_{\omega}(x),f_{\omega}(y))\le R$.
By the definition of $R$ there exists $i_{0}\in\left\{ 1,\dots,t\right\} $
such that $B(x,R)\subset U_{x_{i_{0}}}$. Let $A(0):=[\beta^{i_{0}}]\subset\Sigma$
and $B(0):=\Sigma\setminus[\beta^{i_{0}}]$. We define inductively,
for $j\ge1$, 
\[
A(j):=\left\{ \omega\in B(j-1)\mid\exists i\in\left\{ 1,\dots,t\right\} \,\,\text{such that }B(f_{\omega_{|rj}}(x),R)\subset U_{x_{i}}\text{ and }\left(\omega_{rj+1},\dots\omega_{r(j+1)}\right)=\beta^{i}\right\} 
\]
and $B(j):=B(j-1)\setminus A(j).$ We have 
\[
\left|M_{\vec{p}}^{rn}h(x)-M_{\vec{p}}^{rn}h(y)\right|\le\left|\sum_{j=0}^{n-1}\int_{A(j)}h(f_{\omega_{|rn}}(x))-h(f_{\omega_{|rn}}(y))d\mu_{\vec{p}}\left(\omega\right)\right|+\left|\int_{B(n-1)}h(f_{\omega_{|rn}}(x))-h(f_{\omega_{|rn}}(y))d\mu_{\vec{p}}\left(\omega\right)\right|.
\]
Since $\mu_{\vec{p}}(A(j))\le a^{j}$ for every $j\le n-1$, and $B(f_{\omega_{|r(j+1)}}(x),R)\subset V$,
for every $\omega\in A(j)$, we can estimate with 
\[
S:=\max\left\{ \Lip(f_{i|V^{+}}),\Lip(f_{i|V^{-}})\mid i\in I\right\} <1\quad\text{and }\tilde{c}:=\max\left\{ \eta,S^{r\alpha}\right\} <1,
\]
\begin{align*}
\int_{A(j)}\left|h(f_{\omega_{|rn}}(x))-h(f_{\omega_{|rn}}(y))\right|d\mu_{\vec{p}}\left(\omega\right) & \le a^{j}\sup_{\omega\in A(j)}\left|h(f_{\omega_{|rn}}(x))-h(f_{\omega_{|rn}}(y))\right|\\
 & \le a^{j}\Vert h_{|V}\Vert_{\alpha}\sup_{\omega\in A(j)}\left\{ \left(S^{(n-j-1)r}d\left(f_{\omega_{|r(j+1)}}(x),f_{\omega_{|r(j+1)}}(y)\right)\right)^{\alpha}\right\} \\
 & \le a^{j}S^{(n-j-1)r\alpha}\Lambda^{(j+1)\alpha}\Vert 
 h_{|V}\Vert_{\alpha}d(x,y)^{\alpha}\\
 & = \eta^{j}(S^{r\alpha})^{n-j-1}\Lambda^{\alpha}\Vert 
 h_{|V}\Vert_{\alpha}d(x,y)^{\alpha} \le \tilde{c}^{n-1}\Lambda^{\alpha}\Vert 
 h_{|V}\Vert_{\alpha}d(x,y)^{\alpha}.
\end{align*}
Finally, we verify that 
\begin{align*}
\int_{B(n-1)}\left|h(f_{\omega_{|rn}}(x))-h(f_{\omega_{|rn}}(y))\right|d\mu_{\vec{p}}\left(\omega\right) & \le\mu\left(B(n-1)\right)\sup_{\omega\in B(n-1)}\left|h\circ f_{\omega_{|rn}}(x)-h\circ f_{\omega_{|rn}}(y)\right|\\
 & \le a^{n}\Vert h\Vert_{\alpha}\sup_{\omega\in B(n-1)}d\left(f_{\omega_{|rn}}(x),f_{\omega_{|rn}}(y)\right)^{\alpha}\\
 & \le a^{n}\Vert 
 h\Vert_{\alpha}\Lambda^{n\alpha}d(x,y)^{\alpha}\le\eta^{n}\Vert 
 h\Vert_{\alpha}d(x,y)^{\alpha}.
\end{align*}
We have thus shown that 
\[
\left|M_{\vec{p}}^{rn}h(x)-M_{\vec{p}}^{rn}h(y)\right|\le\left(n\tilde{c}^{n-1}
\Lambda^{\alpha}\Vert h_{|V}\Vert_{\alpha}+\eta^{n}\Vert 
h\Vert_{\alpha}\right)d(x,y)^{\alpha}.
\]
For $n$ sufficiently large, the assertion of the lemma follows. It
is clear that $a=a(\vec{p})$ and thus, $\alpha$ and the other constants
involved, depend continuously on $\vec{p}$. Therefore, the assertion
of the lemma holds with locally uniform constants. The proof is complete. 
\end{proof}
\begin{rem*}
	It follows from the proof that result of the previous lemma holds for every 
	$\alpha<-\log(a)/\log(\Lambda)$. 
\end{rem*}
\begin{thm}
\label{thm:spectralgap}For every $\vec{p}_{0}\in(0,1)^{s}$ there
exists $\delta>0$ and $\alpha>0$ such that $M_{\vec{p}}:\mathcal{C}^{\alpha}(\overline{\R})\rightarrow\mathcal{C}^{\alpha}(\overline{\R})$
has the spectral gap property for every $\vec{p}\in B(\vec{p}_{0},\delta)$.
In particular, for every $\vec{p}\in B(\vec{p}_{0},\delta)$ we have
$T_{\vec{p}}\in\mathcal{C}^{\alpha}(\overline{\R})$ and the convergence
\[
\Vert M_{\vec{p}}^{n}h-(h(\infty)-h(-\infty))T_{\vec{p}}-h(-\infty)\1\Vert_{\alpha}\rightarrow0,\quad\text{as }n\rightarrow\infty,
\]
is exponentially fast. Moreover, the map $\vec{p}\mapsto T_{\vec{p}}\in\mathcal{C}^{\alpha}(\overline{\R})$
is real-analytic on $B(\vec{p}_{0},\delta)$.
\end{thm}

\begin{proof}
By Lemma \ref{lem:doeblin-fortetinequality general} there exist $\delta>0$,
$\alpha>0$,  $n\in\N$, $0<c<1$ and $C>0$  such that, for every $\vec{p}\in 
B(\vec{p}_{0},\delta)$ and for every $h\in \mathcal{C}^{\alpha 
}(\overline{\Bbb{R}})$,
we have the Ionescu-Tulcea and Marinescu inequality 
\[
\Vert M_{\vec{p}}^{n}h\Vert_{\alpha}\le c\Vert h\Vert_{\alpha}+(C+1)\Vert h\Vert_{\infty}.
\]
Therefore, the theorem follows from the well-known result of \cite{MR0037469}
in tandem with the perturbation theory for linear operators (\cite{katoperturbationMR0407617}). 
\end{proof}
\begin{rem*}
	We remark that a result similar to Theorem \ref{thm:spectralgap} has been 
	obtained in \cite[Theorem 
	3.30]{S13Coop} in the framework of random complex dynamical systems. In 
	this 
	paper, we deal with random real one-dimensional dynamical systems. 
	Regarding the 
	proof of Theorem  \ref{thm:spectralgap}, we obtain a simple and 
	straightforward proof by using the Ionescu-Tulcea and 
	Marinescu inequality. 
\end{rem*}
\begin{cor}
For every $\vec{p}_{0}\in(0,1)^{s}$ there exists $\delta>0$ and
$\alpha>0$ such that $\T_{\vec{p}}\subset\mathcal{C}^{\alpha}(\overline{\R})$
for every $\vec{p}\in B(\vec{p}_{0},\delta)$. 
\end{cor}

The next lemma can be proved exactly as in \cite[Lemma 4.1]{JS2016}.
\begin{lem}
For every $n\in\N_{0}^{s}$ we have 
\[
C_{\vec{n}}=M_{\vec{p}}C_{\vec{n}}+\sum_{i=1}^{s}n_{i}\left(C_{\vec{n}-\vec{e}_{i}}\circ f_{i}-C_{\vec{n}-\vec{e}_{i}}\circ f_{s+1}\right).
\]
\end{lem}

The Bernoulli measure $\mu_{\vec{p}}$ on $\Sigma$ defines the probability
measure $\tilde{\mu}_{\vec{p}}=\mu_{\vec{p}}\circ\pi^{-1}$ on $J$
with distribution function 
\[
F_{\vec{p}}:\overline{\R}\rightarrow[0,1],\quad F_{\vec{p}}\left(x\right)=\tilde{\mu}_{\vec{p}}\left\{ (-\infty,x]\right\} .
\]
\begin{lem}
\label{lem:.Fp is fixed point}We have $M_{\vec{p}}F_{\vec{p}}=F_{\vec{p}}$.
\end{lem}

\begin{proof}
Clearly, $\mu_{\vec{p}}$ is a self-similar measure on $\Sigma$,
i.e., 
\[
\mu_{\vec{p}}=\sum_{i=1}^{s+1}p_{i}\mu_{\vec{p}}\circ\sigma_{i}^{-1},
\]
where $\sigma_{i}:\Sigma\rightarrow\Sigma$ is given by $\sigma_{i}(\omega):=i\omega$.
Using that $f_{i}^{-1}\circ\pi=\pi\circ\sigma_{i}$, we obtain for
every Borel set $B\subset\R$, 
\begin{align*}
\tilde{\mu}_{\vec{p}}\left(B\right) & =\mu_{\vec{p}}\left(\pi^{-1}(B)\right)=\sum_{i=1}^{s+1}p_{i}\mu_{\vec{p}}\circ\sigma_{i}^{-1}\left(\pi^{-1}(B)\right)\\
 & =\sum_{i=1}^{s+1}p_{i}\mu_{\vec{p}}\circ\pi^{-1}\left(f_{i}^{-1}\right)^{-1}(B)=\sum_{i=1}^{s+1}p_{i}\tilde{\mu}_{\vec{p}}\left(f_{i}(B)\cap J\right).
\end{align*}
Setting $B:=(-\infty,x]$ the lemma follows.
\end{proof}

Note that $M_{\vec{p}}$ can be defined for every Borel measurable function. 
\begin{lem}
\label{lem:unqiueness fixedpoint}$T_{\vec{p}}$ is the unique bounded
Borel measurable function such that $M_{\vec{p}}T_{\vec{p}}=T_{\vec{p}}$
and $T_{\vec{p}|V^{-}}=0$ and $T_{\vec{p}|V^{+}}=1$.  In particular,
$T_{\vec{p}}=F_{\vec{p}}$.
\end{lem}

\begin{proof}
By \eqref{eq:Mp
	convergence measurable} we have that $M_{\vec{p}}T_{\vec{p}}=T_{\vec{p}}$.
Clearly, $T_{\vec{p}}$ is bounded and measurable, and satisfies $T_{\vec{p}|V^{-}}=0$
and $T_{\vec{p}|V^{+}}=1$. Let $h$ be another bounded Borel measurable
function such that $M_{\vec{p}}h=h$ and $h_{|V^{-}}=0$ and $h_{|V^{+}}=1$.
Note that  (\ref{eq:Mp convergence measurable}) in fact holds for every bounded 
Borel measurable function which is continuous at $\{\pm \infty \}$. Therefore, 
we have 
\[
h(x)=\lim_{n\rightarrow\infty}M_{\vec{p}}^{n}h(x)=h(\infty)T_{\vec{p}}(x)+h(-\infty)(1-T_{\vec{p}}(x))=T_{\vec{p}}(x),
\]
which proves the asserted uniqueness.  To prove that $T_{\vec{p}}=F_{\vec{p}}$
we note that $F_{\vec{p}}$ is bounded and Borel measurable, $M_{\vec{p}}F_{\vec{p}}=F_{\vec{p}}$,
$F_{\vec{p}|V^{-}}=0$ and $F_{\vec{p}|V^{+}}=1$. The assertion of
the lemma follows. 
\end{proof}
The following fact follows immediately from the definition of $F_{\vec{p}}$. 
\begin{fact}
\label{fact:variespricisely}$F_{\vec{p}}$ (and hence, $T_{\vec{p}}$)
is locally constant precisely on $\R\setminus J$.
\end{fact}

\subsection{Improved spectral gap property for systems with separating condition\label{subsec:Improved-spectral-gap}}

In this section we derive an improved spectral gap property for systems
satisfying the separating condition. We define the potentials 
\[
\varphi:\Sigma\rightarrow\R,\quad\varphi(\omega):=-\log\left|f_{\omega_{1}}'(\pi(\omega))\right|,\quad\text{and}\quad\psi:=\psi_{\vec{p}}:\Sigma\rightarrow\R,\quad\psi(\omega):=\log
 p_{\omega_{1}}.
\]
We define the shift map $\sigma:\Sigma\rightarrow\Sigma$, $\sigma((\omega_{1},\omega_{2},\dots)):=(\omega_{2},\omega_{3},\dots)$.
For $u:\Sigma\rightarrow\R$ and $n\in\N$ we denote by $S_{n}u:=\sum_{k=0}^{n-1}u\circ\sigma^{k}$
the $n$th ergodic sum. Further we let 
\begin{equation} \label{defalphaplusminus}
\alpha_{-}:=\alpha_{-}(\vec{p}):=\inf_{\omega\in\Sigma}\liminf_{n\rightarrow\infty}\frac{S_{n}\psi_{\vec{p}}(\omega)}{S_{n}\varphi(\omega)},\quad\alpha_{+}:=\alpha_{+}(\vec{p}):=\sup_{\omega\in\Sigma}\limsup_{n\rightarrow\infty}\frac{S_{n}\psi_{\vec{p}}(\omega)}{S_{n}\varphi(\omega)},
\end{equation}
and we refer to $\alpha_{-}$ as the bottom of the spectrum. 
\begin{lem}
\label{lem:semicontinuity of boundary points}The map $\vec{p}\mapsto\alpha_{-}(\vec{p})$
is lower semi-continuous on $(0,1)^{s}$. 
\end{lem}

\begin{proof}
Let $\vec{p}_{0}\in(0,1)^{s}$. For every $\epsilon>0$ there exists
$\delta>0$ such that $\Vert\psi_{\vec{p}}-\psi_{\vec{p}_{0}}\Vert<\epsilon/\min|\varphi|$
for every $\vec{p}$ with $|\vec{p}-\vec{p}_{0}|<\delta$. Let $\vec{p}$
with $|\vec{p}-\vec{p}_{0}|<\delta$, $\omega\in\Sigma$ and 
$n_{k}\rightarrow\infty$, as $k\rightarrow \infty$,
such that 
\[
\lim_{k\rightarrow\infty}\frac{S_{n_{k}}\psi_{\vec{p}}(\omega)}{S_{n_{k}}\varphi(\omega)}=\alpha_{-}(\vec{p}).
\]
The existence of such   $\omega\in\Sigma$ and 
$(n_{k})$ follows from \cite{MR1738952}. We have 
\[
\frac{S_{n_{k}}\psi_{\vec{p}}(\omega)}{S_{n_{k}}\varphi(\omega)}=\frac{S_{n_{k}}\psi_{\vec{p}_{0}}(\omega)}{S_{n_{k}}\varphi(\omega)}+\frac{S_{n_{k}}\psi_{\vec{p}}(\omega)-S_{n_{k}}\psi_{\vec{p}_{0}}(\omega)}{S_{n_{k}}\varphi(\omega)}.
\]
Since 
\[
\left|\frac{S_{n_{k}}\psi_{\vec{p}}(\omega)-S_{n_{k}}\psi_{\vec{p}_{0}}(\omega)}{S_{n_{k}}\varphi(\omega)}\right|\le\frac{n_{k}\Vert\psi_{\vec{p}}-\psi_{\vec{p}_{0}}\Vert}{n_{k}\min|\varphi|}<\epsilon\quad\text{and }\quad\liminf_{k\rightarrow\infty}\frac{S_{n_{k}}\psi_{\vec{p}_{0}}(\omega)}{S_{n_{k}}\varphi(\omega)}\ge\alpha_{-}(\vec{p}_{0}),
\]
 we conclude that 
\[
\alpha_{-}(\vec{p})=\lim_{k\rightarrow\infty}\frac{S_{n_{k}}\psi_{\vec{p}}(\omega)}{S_{n_{k}}\varphi(\omega)}\ge\liminf_{k\rightarrow\infty}\frac{S_{n_{k}}\psi_{\vec{p}_{0}}(\omega)}{S_{n_{k}}\varphi(\omega)}-\epsilon\ge\alpha_{-}(\vec{p}_{0})-\epsilon.
\]
\end{proof}
\begin{rem*}
Similarly, one can show that $\vec{p}\mapsto\alpha_{+}(\vec{p})$
is upper semi-continuous on $(0,1)^{s}$. 
\end{rem*}
\begin{lem}
\label{lem:uniform contraction on unbounded fatou}There exists $\eta<1$
such that for every compact set $K\subset\overline{\R}\setminus J$
there exists a constant $C_{K}<\infty$ such that for all $x,y\in K$
belonging to the same connected component of $\overline{\R}\setminus J$,
we have for every $\omega\in I^{*}$, 
\[
d(f_{\omega}(x),f_{\omega}(y))\le C_{K}\eta^{\left|\omega\right|}d(x,y).
\]
\end{lem}

\begin{proof}
There exists $N=N(K)\in\N$ such that for all $n\ge N$, $\omega\in I^{n}$
and all $x,y\in K$ belonging to the same component of $\overline{\R}\setminus J$,
we have either $f_{\omega}(x),f_{\omega}(y)\in V^{+}$ or  
$f_{\omega}(x),f_{\omega}(y)\in V^{-}$.
For otherwise, there exists $\tau\in \Sigma$ such that 
$f_{\tau_{|n}}(x)\in\overline{\R}\setminus V$,
for all $n\in\N$, contradicting that $x\notin J$. Let $S:=\max\left\{ \Lip(f_{i|V^{+}}),\Lip(f_{i|V^{-}})\mid i\in I\right\} $.
Since $(f_{i})_{i\in I}$ is contracting near infinity, we have $S<1$.
For $n\ge N$ and $\omega\in I^{n}$ we have 
\begin{align*}
d\left(f_{\omega}(x),f_{\omega}(y)\right) & \le S^{n-N}d\left(f_{\omega_{|N}}(x),f_{\omega_{|N}}(y)\right)\le S^{n-N}\left(\max_{i\in I}\Lip(f_{i})\right)^{N}d\left(x,y\right)=S^{n}C_{K}d(x,y),
\end{align*}
where we have set $C_{K}:=S^{-N}\left(\max_{i\in I}\Lip(f_{i})\right)^{N}$.
The estimate for $n<N$ can be shown similarly. 
\end{proof}
\begin{defn}
\label{def:sepcond}We say that $(f_{i})_{i\in I}$ satisfies the
\emph{separating condition}  if there exists a non-empty bounded open
interval  $O\subset\R$ such that $f_{i}^{-1}(O)\subset O$, for all
$i\in I$, and for all $i,j\in I$ with $i\neq j$, we have $f_{i}^{-1}(\overline{O})\cap f_{j}^{-1}(\overline{O})=\emptyset.$
If the separating condition holds, then we may always assume that
$O$ is bounded and that $J\subset O$. 
\end{defn}

The proof of the following lemma is standard and therefore omitted
(see e.g., \cite{MR2003772}).
\begin{lem}
[Bounded distortion] \label{lem:bounded distortion}Let $\Omega$
be a bounded open set such that $f_{i}^{-1}(\Omega)\subset\Omega$
for all $i\in I$. Then we have 
\begin{equation}
D:=D(\Omega):=\sup\left\{ \frac{\left|f'_{\gamma}(v_{1})\right|}{\left|f'_{\gamma}(v_{2})\right|}\Big|\gamma\in I^{*},v_{1},v_{2}\in f_{\gamma}^{-1}(\overline{\Omega})\right\} <\infty.\label{eq:bdd-distortion}
\end{equation}
\end{lem}

\begin{lem}
\label{lem: doeblin fortet inequality}Suppose that $(f_{i})_{i\in I}$
satisfies the separating condition. Let $\vec{p}_{0}\in(0,1)^{s}$
and $0<\alpha<\alpha_{-}(\vec{p}_{0})$. Then there exist $\delta>0$,
$n\in\N$ and constants $0<c<1$ and $C>0$ such that,\,\,\,for every $\vec{p}\in 
B(\vec{p}_{0},\delta)$,
$h\in\mathcal{C}^{\alpha}(\overline{\R})$ and for all $x,y\in\overline{\R}$,
\[
\left|M_{\vec{p}}^{n}h(x)-M_{\vec{p}}^{n}h(y)\right|\le\left(c\Vert h\Vert_{\alpha}+C\Vert h\Vert_{\infty}\right)d(x,y)^{\alpha}.
\]
\end{lem}

\begin{proof}
Suppose that the separating condition holds with bounded open interval
$O$. We may assume that $J\subset O$ and that there exists $r_{0}>0$
such that for all $i,j\in I$ with $i\neq j$, 
\[
\overline{f_{i}^{-1}(B(O,r_{0}))}\subset B(O,r_{0})\quad\text{and }f_{j}\left(\overline{f_{i}^{-1}(B(O,r_{0}))}\right)\subset\overline{\R}\setminus J,
\]
where $B(O,r_{0}):=\bigcup_{u\in O}B(u,r_{0})$. For $i\neq j$ we
define the compact sets 
\[
K_{i,j}:=f_{j}\left(\overline{f_{i}^{-1}(B(O,r_{0}))}\right)\quad\text{and}\quad K':=\bigcup_{i,j\in I,\,\,i\neq j}K_{i,j}\subset\overline{\R}\setminus J.
\]
Let $\delta:=d(\overline{\R}\setminus O,J)$. Since $J\subset O$
is compact, we have $\delta>0$. Define the compact set 
\[
K:=K'\cup\left\{ u\in\overline{\R}\setminus J,d(u,J)\ge\delta/2\right\} \subset\overline{\R}\setminus J.
\]
Since $\overline{O}$ is compact, by modifying $r_{0}$ if necessary,
we may assume that 
\begin{equation}
\forall x'\in\overline{O}\,\,\forall y'\in B(x',r_{0})\,\,\forall i\in I:\,\,d(f_{i}(x'),f_{i}(y'))<\delta/2.\label{eq:uniform-continuous delta}
\end{equation}

Let $x,y\in\overline{\R}$ and $n\in\N$ sufficiently large (to be
determined later). We now distinguish two cases. First suppose that
$x\notin O$.  We may assume that $d(x,y)\le\delta/2$. Hence, $x$
and $y$ are contained in the compact set $\left\{ u\in\overline{\R}\mid d(u,J)\ge\delta/2\right\} \subset K\subset\overline{\R}\setminus J$
and belong to the same connected component of $\overline{\R}\setminus J$.
Therefore, by Lemma \ref{lem:uniform contraction on unbounded fatou},
we have 
\[
\left|M_{\vec{p}}^{n}h(x)-M_{\vec{p}}^{n}h(y)\right|\le\sum_{\left|\tau\right|=n}p_{\tau}\left|h(f_{\tau}(x))-h(f_{\tau}(y))\right|\le\sum_{\left|\tau\right|=n}p_{\tau}\Vert h\Vert_{\alpha}\left(C_{K}\eta^{n}\right)^{\alpha}d(x,y)^{\alpha}\le\Vert h\Vert_{\alpha}\left(C_{K}\eta^{n}\right)^{\alpha}d(x,y)^{\alpha}.
\]
For $n$ sufficiently large, we have $\left(C_{K}\eta^{n}\right)^{\alpha}<1$.
This finishes the proof in the case when $x\notin O$. 

Next we consider the remaining case $x\in O$.  Let 
\[
\ell_{x}:=\sup\left\{ \ell\ge0\mid\exists\omega\in I^{*} \cup \{ \emptyset \} 
\text{ such that }f_{\omega_{|j}}(x)\in O\,\,\text{ for any }0\le 
j\le\ell\right\} \in\N\cup\{0,\infty\}.
\]
Then there exists a unique $(\omega _{1},\ldots, \omega _{\ell_{x}})\in I^{\ell_{x}}$  such that $f_{\omega _{\ell_{x}}}\cdots f_{\omega _{1}}(x)\in O.$
Here, if $\ell_{x}=0$, then we set $f_{\omega _{\ell_{x}}}\cdots f_{\omega 
_{1}}(x)=x$.
We distinguish two subcases (a) $\ell_{x}\ge n$ and (b) $\ell_{x}<n$.
We begin with subcase (a). For all
$j=0,1,2,\dots$ let 
\[B_{j}(x):= f_{\omega |_{j}}^{-1}(B(f_{\omega |_{j}}(x),r_{0})),\]
 where 
$B_{0}(x):=B(x,r_{0}).$ 
We may assume that $y\in B_{n}(x)$. We have the decomposition 
\[
I^{n}=\bigcup_{j\le n}\bigcup_{\left|\tau\right|=n-j,\tau_{1}\neq\omega_{j+1}}\left\{ \omega_{|j}\tau\right\} .
\]
For all $j<n$ and $\tau_{1}\in I\setminus\left\{ \omega_{j+1}\right\} $
we have 
\[
f_{\tau_{1}}\left(f_{\omega_{|j}}\left(B_{n}(x)\right)\right)\subset f_{\tau_{1}}\left(f_{\omega_{j+1}}^{-1}(O,r_{0})\right)\subset K'\subset\overline{\R}\setminus J.
\]
By Lemma \ref{lem:uniform contraction on unbounded fatou}, for all $j<n$ and
$\tau \in I^{n-j}$ with $\tau _{1} \neq \omega _{j+1}$, we have
\begin{align*}
d\left(f_{\tau}\left(f_{\omega_{|j}}(x)\right),f_{\tau}\left(f_{\omega_{|j}}(y)\right)\right) & \le C_{K'}\eta^{n-j-1}d\left(f_{\tau_{1}}(f_{\omega_{|j}}(x)),f_{\tau_{1}}(f_{\omega_{|j}}(y))\right)\\
 & \le C_{K'}\Lip(f_{\tau_{1}})\eta^{n-j-1}d\left(f_{\omega_{|j}}(x),f_{\omega_{|j}}(y)\right).
\end{align*}
Put $C_{1}:=C_{K'}\max_{i\in I}\Lip(f_{i})$. Since $h$ is H\"older
continuous, we can assert that 
\begin{align*}
\left|M_{\vec{p}}^{n}h(x)-M_{\vec{p}}^{n}h(y)\right| & =\left|\sum_{j\le n}p_{\omega|_{j}}\sum_{\left|\tau\right|=n-j,\tau_{1}\neq\omega_{j+1}}p_{\tau}\left(h\left(f_{\tau}\left(f_{\omega_{|j}}(x)\right)\right)-h\left(f_{\tau}\left(f_{\omega_{|j}}(y)\right)\right)\right)\right|\\
 & \le\sum_{j\le n}p_{\omega|_{j}}\sum_{\left|\tau\right|=n-j,\tau_{1}\neq\omega_{j+1}}p_{\tau}\Vert h\Vert_{\alpha}C_{1}^{\alpha}(\eta^{n-j-1})^{\alpha}d\left(f_{\omega_{|j}}(x),f_{\omega_{|j}}(y)\right)^{\alpha}\\
 & \le C_{1}^{\alpha}\Vert h\Vert_{\alpha}\sum_{j\le 
 n}p_{\omega|_{j}}\left(\eta^{n-j-1}\right)^{\alpha}d\left(f_{\omega_{|j}}(x),f_{\omega_{|j}}(y)\right)^{\alpha}.
\end{align*}
Using that $d$ is strongly equivalent to the Euclidean metric on
the compact set $\overline{B(O,r_{0})}\subset\R$ and combining with
the bounded distortion estimate in (\ref{eq:bdd-distortion}) of Lemma
\ref{lem:bounded distortion} with $\Omega:=B(O,r_{0})$, we deduce
the existence of $D_{0}<\infty$ such that 
\[
\left|M_{\vec{p}}^{n}h(x)-M_{\vec{p}}^{n}h(y)\right|\le C_{1}^{\alpha}D_{0}^{\alpha}\Vert h\Vert_{\alpha}\sum_{j\le n}p_{\omega|_{j}}\left(\eta^{n-j-1}\right)^{\alpha}|f_{\omega_{|j}}'(x)|^{\alpha}d(x,y)^{\alpha}.
\]
The following was proved in \cite[(6.2) in the proof of Theorem 1.3]{JS2016}.
There exists a $C(\varphi,\psi_{\vec{p}})$, which depends
continuously on $\vec{p}\in (0,1)^{s}$, such that for all $j\in\N$
and for all $\omega\in\Sigma$, 
\begin{equation}
\e^{S_{j}\psi_{\vec{p}}(\omega)}\le C(\varphi,\psi_{\vec{p}})\e^{\alpha_{-}(\vec{p})S_{j}\varphi(\omega)}.\label{eq:uniform psi phi estimate}
\end{equation}
Let $\omega=(\omega_{1},\dots,\omega_{\ell_{x}},\omega_{1},\dots,\omega_{\ell_{x}},\dots)\in\Sigma$.
By combining (\ref{eq:uniform psi phi estimate}) with the bounded
distortion estimate in (\ref{eq:bdd-distortion}) we have for some
$D_{1}$ with $D_{1}>D_{0}$ 
\[
p_{\omega_{|j}}\le 
D_{1}^{\alpha_{-}(\vec{p})}C(\varphi,\psi_{\vec{p}})|f_{\omega_{|j}}'(x)|^{-\alpha_{-}(\vec{p})}.
\]
Hence, we obtain 
\begin{align*}
\left|M_{\vec{p}}^{n}h(x)-M_{\vec{p}}^{n}h(y)\right| & \le C_{1}^{\alpha}D_{1}^{\alpha+\alpha_{-}(\vec{p})}C(\varphi,\psi_{\vec{p}})\Vert h\Vert_{\alpha}\sum_{j\le n}\eta^{\alpha(n-j-1)}|f_{\omega_{|j}}'(x)|^{\alpha-\alpha_{-}(\vec{p})}d(x,y)^{\alpha}.
\end{align*}

For $\delta>0$ sufficiently small and $\vec{p}\in B(\vec{p}_{0},\delta)$
we have $\sup_{\vec{p}\in B(\vec{p}_{0},\delta)}(\alpha-\alpha_{-}(\vec{p}))<0$
by Lemma \ref{lem:semicontinuity of boundary points}. Also we can
define $C(\varphi,\psi):=\sup_{\vec{p}\in B(\vec{p}_{0},\delta)}C(\varphi,\psi_{\vec{p}})<\infty$.
Since $|f_{\omega_{|j}}'(x)|\ge\lambda^{j}$ there exist $\tilde{\eta}<1$
and $C_{2}<\infty$ such that for all $j\le n$ and $\vec{p}\in B(\vec{p}_{0},\delta)$,
\[
\eta^{\alpha\left(n-j-1\right)}|f_{\omega_{|j}}'(x)|^{\alpha-\alpha_{-}(\vec{p})}\le C_{2}\tilde{\eta}^{n}.
\]
Therefore, 
\[
\left|M^{n}h(x)-M^{n}h(y)\right|\le C_{1}^{\alpha}D_{1}^{\alpha+\alpha_{-}(\vec{p})}C(\varphi,\psi)\Vert h\Vert_{\alpha}C_{2}n\tilde{\eta}^{n}d(x,y)^{\alpha}.
\]
Put $c:=C_{1}^{\alpha}D_{1}^{\alpha+\alpha_{-}(\vec{p})}C(\varphi,\psi)\Vert h\Vert_{\alpha}C_{2}n\tilde{\eta}^{n}$.
For $n$ sufficiently large we have $c<1$. Thus, assuming subcase (a),
we have derived the desired estimate.

Finally, to complete the proof, let us consider the subcase (b) when
$\ell_{x}<n$. We may assume that $y\in B_{\ell_{x}}(x)$. We estimate
\begin{align*}
\left|M_{\vec{p}}^{n}h(x)-M_{\vec{p}}^{n}h(y)\right| & \le\left|\sum_{j<\ell_{x}}p_{\omega|_{j}}\sum_{\left|\tau\right|=n-j,\tau_{1}\neq\omega_{j+1}}p_{\tau}\left(h\left(f_{\tau}\left(f_{\omega_{|j}}(x)\right)\right)-h\left(f_{\tau}\left(f_{\omega_{|j}}(y)\right)\right)\right)\right|\\
 & \quad\quad+p_{\omega|_{\ell_{x}}}\left|\sum_{\left|\tau\right|=n-\ell_{x}}p_{\tau}\left(h\left(f_{\tau}\left(f_{\omega|_{\ell_{x}}}(x)\right)\right)-h\left(f_{\tau}\left(f_{\omega|_{\ell_{x}}}(y)\right)\right)\right)\right|.
\end{align*}
The first summand on the right-hand side satisfies an $\alpha$-H\"older
condition with $c<1$ for $n$ large by the same arguments as in (a)
above. Finally, to deal with the second summand, let $x':=f_{\omega|_{\ell_{x}}}(x)$
and $y':=f_{\omega|_{\ell_{x}}}(y)$. Since $y\in B_{\ell_{x}}(x)$
we have that $d(x',y')<r_{0}$. By the definition of $\ell_{x}$ we
have that $f_{\tau_{1}}(x')\notin O$ for all $\tau_{1}\in I$. By
(\ref{eq:uniform-continuous delta}) we have $d(f_{\tau_{1}}(y'),J)\ge\delta/2>0$.
Since moreover $f_{\tau_{1}}(x')$ and $f_{\tau_{1}}(y')$ are in
the same component of $\overline{\R}\setminus J$, Lemma \ref{lem:uniform contraction on unbounded fatou}
implies that 
\[
d\left(f_{\tau}\left(f_{\omega|_{\ell_{x}}}(x)\right),f_{\tau}\left(f_{\omega|_{\ell_{x}}}(y)\right)\right)\le C_{K}\eta^{n-\ell_{x}}d(f_{\omega|_{\ell_{x}}}(x),f_{\omega|_{\ell_{x}}}(y)).
\]
The rest of the proof runs as in subcase (a) above. The proof is complete.
\end{proof}
We thus obtain the following strengthening of Theorem \ref{thm:spectralgap}
when the separating condition holds.
\begin{thm}
\label{thm: spectral gap separating condition}Suppose that $(f_{i})_{i\in I}$
satisfies the separating condition.  Let $\vec{p}_{0}\in(0,1)^{s}$
and $\alpha<\alpha_{-}(\vec{p}_{0})$. There exists $\delta>0$ such
that $M_{\vec{p}}:\mathcal{C}^{\alpha}(\overline{\R})\rightarrow\mathcal{C}^{\alpha}(\overline{\R})$
has the spectral gap property for every $\vec{p}\in B(\vec{p}_{0},\delta)$.
Moreover, the map $\vec{p}\mapsto T_{\vec{p}}\in\mathcal{C}^{\alpha}(\overline{\R})$
is real-analytic on $B(\vec{p}_{0},\delta)$. 
\end{thm}

\begin{cor}
Suppose that $(f_{i})_{i\in I}$ satisfies the separating condition.
Let $\vec{p}_{0}\in(0,1)^{s}$ and $\alpha<\alpha_{-}(\vec{p}_{0})$.
There exists $\delta>0$ such that $\T_{\vec{p}}\subset\mathcal{C}^{\alpha}(\overline{\R})$
for every $\vec{p}\in B(\vec{p}_{0},\delta)$. 
\end{cor}

\begin{rem}
\label{rem:sharp alpha minus}By Proposition \ref{prop:derivatives are not alpha- hoelder}
below it will turn out that in many cases, $\T_{p}\nsubseteq\mathcal{C}^{\alpha_{-}(p)}(\overline{\R})$.
Hence, by the previous corollary, the spectral gap property stated
in Theorem \ref{thm: spectral gap separating condition} is sharp.
In particular, this is the case for the classical Takagi function.
Namely, let $f_{1}(x)=2x$, $f_{2}(x)=2x-1$. Then $J=[0,1]$ and
$T_{1/2}(x)=x$ for $x\in J$ and $C_{1}$ is a multiple of the classical
Takagi function on $J$. Hence, $\alpha_{-}=1$, $T_{1/2}$ is Lipschitz
continuous, and $C_{1}$ is $\alpha$-H\"older continuous for every
$\alpha<1$, but not Lipschitz continuous. For further examples, see
Proposition \ref{prop:derivatives are not alpha- hoelder}. 
\end{rem}

\section{Multifractal Analysis of the Pointwise H\"older exponent } \label{section:Multifractal}

In this section we perform a complete multifractal analysis of the
pointwise H\"older exponents of the elements of $\T$ associated
with $(f_{i})_{i\in I}$ and $\vec{p}\in(0,1)^{s}$. We begin by providing
the necessary terminology which has been introduced in \cite{JS2016}.
We use $\vec{n}=(n_{1},\dots,n_{s})$ to denote an element of $\N_{0}^{s}$.
Let $\vec{e}_{i}\in\N_{0}^{s}$ denote the element whose $i$th component
is $1$ and all other components are $0$. An  element $A\in \R ^{\N_{0}^{s}\times \N_{0}^{s}}$ is  represented as $A=(A_{\vec{x},\vec{y}})_{(\vec{x},\vec{y})\in \N_{0}^{s}\times \N_{0}^{s}}$,  where $A_{\vec{x},\vec{y}}\in \R $,    and  such an element $A$ is called an ($\N_{0}^{s}$-)matrix.  
$A_{\vec{x},\vec{y}}$ is called the $(\vec{x},\vec{y})$-component of $A.$ We 
denote by $\1_{\vec{n},\vec{m}}\in\R^{\N_{0}^{s}\times\N_{0}^{s}}$
the matrix such that for every $(\vec{x},\vec{y})\in\N_{0}^{s}\times\N_{0}^{s}$
the $(\vec{x},\vec{y})$-component of $\1_{\vec{n},\vec{m}}$ is given
by 
\[
(\1_{\vec{n},\vec{m}})_{\vec{x},\vec{y}}=\begin{cases}
1, & \vec{x}=\vec{n},\,\,\vec{y}=\vec{m}\\
0, & \text{else}.
\end{cases}
\]

In order to investigate $\T$ we define the matrix cocycle $A_{0}:\Sigma\times\N\rightarrow\R^{\N_{0}^{s}\times\N_{0}^{s}}$
given by 
\[
A_{0}(\omega,1):=\begin{cases}
\sum_{\vec{n}\in\N_{0}^{s}}(p_{\omega_{1}}\1_{\vec{n},\vec{n}}+n_{\omega_{1}}\1_{\vec{n},\vec{n}-\vec{e}_{\omega_{1}}}), & \omega_{1}\in\{1,\dots,s\}\\
\sum_{\vec{n}\in\N_{0}^{s}}(p_{\omega_{1}}\1_{\vec{n},\vec{n}}-\sum_{i=1}^{s}n_{i}\1_{\vec{n},\vec{n}-\vec{e}_{i}}), & \omega_{1}=s+1
\end{cases}
\]
and for $k\in\N$, 
\[
A_{0}(\omega,k):=A_{0}(\omega,1)A_{0}(\sigma\omega,1)\dots A_{0}(\sigma^{k-1}\omega,1)\in\R^{\N_{0}^{s}\times\N_{0}^{s}}.
\]

Here,  the matrix product $A_{0}(\tau,1) 
A_{0}(\upsilon,1)\in\R^{\N_{0}^{s}\times\N_{0}^{s}}$ is for $\tau,\upsilon\in 
\Sigma$ and $\vec{l},\vec{m}\in\N_{0}^{s}$ given by  \begin{equation} 
\left(A_{0}(\tau,1)\cdot 
A_{0}(\upsilon,1)\right)_{\vec{l},\vec{m}}:=\sum_{\vec{k}\in\N_{0}^{s}}\left(A_{0}(\tau,1)\right)_{\vec{l},\vec{k}}\cdot\left(A_{0}(\upsilon,1)\right)_{\vec{k},\vec{m}}.\label{eq:definition-matrix
 product} \end{equation}

Note that the sum in (\ref{eq:definition-matrix product}) is actually
a finite sum. Further matrix products in the definition of $A_{0}(\omega ,k)$
are defined in the same way. We also define 
\[
A(\omega,k):=(p_{\omega_{|k}})^{-1}A_{0}(\omega,k)\in\R^{\N_{0}^{s}\times\N_{0}^{s}}.
\]
Moreover, for all $a,b\in\R$ we define the matrix
\[
U(a,b):=(u_{\vec{n}}(a,b))_{\vec{n}\in\N_{0}^{s}}\in\R^{\N_{0}^{s}}\quad\text{given by }u_{\vec{n}}(a,b):=C_{\vec{n}}(a)-C_{\vec{n}}(b).
\]

\begin{rem*} In (\ref{eq:definition-matrix product}) and in the following we 
make use of the product of matrices with an infinite index set. These matrix 
products will always be well defined, since either the first factor of the 
product possesses at most finitely many non-zero entries in each row, or the 
second factor contains at most finitely many non-zero entries in each column. 
\end{rem*}
Since the above definitions of $A_0$ and $A$ coincide with the ones given in 
\cite{JS2016} we 
immediately obtain the following two lemmas. For 
$\vec{n},\vec{m}\in \N_0^s $ we write $\vec{n}\le \vec{m}$ if $n_i\le m_i$ 
for each $1\le i\le s$. 
\begin{lem}[\cite{JS2016}, Lemma 4.5]
	\label{lem:idnm}
	Let $\omega \in \Sigma$ and  $k\in \N .$ Then 
	$A(\omega ,k)_{\vec{n},\vec{n}}=1$ for every $\vec{n}\in \N_{0}^{s}.$ 
	Also, $A(\omega ,k)_{\vec{n},\vec{m}}=0$ unless $\vec{m}\le \vec{n}.$ 
\end{lem}
%
%
%
%
\begin{lem}[\cite{JS2016}, Lemma 4.8]
\label{lem:growth of cocycle} 
Let $\omega\in I^{\N}$ and $k\in\N$. Put
$m_{i}:=m_{i}(k):=\card\left\{ 1\le j\le k:\omega_{j}=i\right\} $
for $1\le i\le s+1$. 
Let $\vec{m}=(m_{i})_{i=1}^{s}\in \N_{0}^{s}.$ 
Let $\vec{q}$, $\vec{r}\in \N_0^s$ with  $\vec{0}\le\vec{r}\le\vec{q}$. Then 
%
%
there exists a constant $K\ge1$
which depends on $\vec{q}$ and the probability vector $\vec{p}$ 
but not on $k$ 
such that 
\[
\left|A(\omega,k)_{\vec{q},\vec{r}}\right|\le 
K\left(\prod_{i=1}^{s}\tilde{m}_{i}^{q_{i}-r_{i}}\right)\tilde{m}_{s+1}^{|\vec{q}|-|\vec{r}|}
\mbox{ and }
\left|A(\omega ,k)_{\vec{q},\vec{r}}\right| \le Kk^{|\vec{q}|},
\]
where  $\tilde{m}_{j}:=\max\{ 1, m_{j}\}$ for $1\leq j\leq s+1.$
If $\omega_{j}\neq s+1$ for all $1\leq j\leq k$ 
and $m_{i}>q_{i}-r_{i}$ for all $1\le i\le s$,
then there exists $K'>0$ depending only on $\vec{q}$  such that
\[
A(\omega,k)_{\vec{q},\vec{r}}\ge K'\prod_{i=1}^{s}m_{i}^{q_{i}-r_{i}}.
\]
%
\end{lem}

\begin{defn}
\label{def:openset}We say that $(f_{i})_{i\in I}$ satisfies the
open set condition if there exists a non-empty bounded open interval
$O\subset\R$ such that $f_{i}^{-1}(O)\subset O$, for all $i\in I$,
and for all $i,j\in I$ with $i\neq j$ we have $f_{i}^{-1}(O)\cap f_{j}^{-1}(O)=\emptyset$. 
\end{defn}

We write $A\le B$ for subsets $A,B\subset\R$ if $a\le b$ for every
$a\in A$ and $b\in B$. 
\begin{rem*}
If $(f_{i})_{i\in I}$ satisfies the open set condition, then we will
always assume that $f_{i}^{-1}(\overline{O})\le f_{j}^{-1}(\overline{O})$
for all $i,j\in I$ with $i<j$.
\end{rem*}
The purpose of the above definitions is the following.
\begin{lem}
\label{lem:functional equation via cocycle}Suppose that $(f_{i})_{i\in I}$
satisfies the open set condition.  Let $k\in\N$, $\omega\in I^{k}$
and $x,y\in f_{\omega}^{-1}(\overline{O})$. Then we have 
$U(x,y)=A_{0}(\overline{\omega},k)U(f_{\omega}(x),f_{\omega}(y))$. Here, we set
$\overline{\omega }:= (\omega _{1}\dots\omega_{k}, \omega _{1}\dots\omega
_{k}\dots)\in \Sigma$.
\end{lem}

\begin{proof}
The assertion can be shown as in \cite[Lemma 4.7]{JS2016} if we observe
that $u_{\vec{n}}(f_{\tau}(x),f_{\tau}(y))=0$, for all $\tau\in I^{k}$
with $\tau\neq\omega$ and for all $\vec{n}\in\N_{0}^{s}$. To prove
this, note that by the open set condition we have that $[f_{\tau}(x),f_{\tau}(y)]\cap\overline{O}$
has at most one point. Consequently, $T_{\vec{p}}$ is constant on
$[f_{\tau}(x),f_{\tau}(y)]$ and thus every $C_{\vec{n}}$ is constant
on $[f_{\tau}(x),f_{\tau}(y)]$, $\vec{n}\in\N_{0}^{s}$. 
\end{proof}
Recall that $G=\left\langle f_{1},\dots,f_{s+1}\right\rangle $ and
write $G(x):=\left\{ g(x)\mid g\in G\right\} $.
\begin{lem}
\label{lem:L is perfect}Suppose that $(f_{i})_{i\in I}$ satisfies
the open set condition.  Let $x_{0}\in J$. Then there exist $a,b\in\left(J\cap 
O\right)\setminus G(x_{0})$
 such that $b$ is arbitrarily close
to $a$ and $T_{\vec{p}}(a)\neq T_{\vec{p}}(b)$. 
\end{lem}
\begin{proof} First recall the well-known fact that $J$ is a non-empty perfect 
subset of $\R$, since $\{ f_{1},\ldots, f_{s+1}\} $
does
not have a common fixed point in $\Bbb{R}$. Thus, every neighborhood of any 
point of  $J$ contains 
uncountably many points in $J$. Let $x_0 \in J$ and  $\epsilon>0$.  Since 
$J\subset \overline{O}$ and $G(x_0)$ is countable, there exists $a\in (J\cap 
O)\setminus G(x_0)$. Since $T_{\vec{p}}$ 
is not locally constant at $a$ and
$T_{p}$ is locally constant on $\overline{\Bbb{R}}\setminus J$ by Fact	
\ref{fact:variespricisely}, there 
exists $c\in B(a,\epsilon/2)\cap J\cap O$ such that $T_{\vec{p}}(c)\neq 
T_{\vec{p}}(a)$. Finally, since $T_{\vec{p}}$ is continuous at $c$ by Theorem 
\ref{thm:spectralgap}, $G(x_0)$ is countable and every neighborhood of $c$ in 
$J$  contains 
uncountably many points, there exists $b \in (B(c,\epsilon/2)\cap J \cap O)\setminus 
G(x_0)$ such that  $T_{\vec{p}}(b)\neq T_{\vec{p}}(a)$. Clearly, we also have 
$b\in B(a,\epsilon)\cap  O$.
\end{proof}


By using Lemma \ref{lem:L is perfect}
we can extend the methods used in \cite{JS2016}. The following lemma is the 
analogue of  \cite[Lemma 4.9]{JS2016}.
\begin{lem}
	\label{lem:growth of solutions}Suppose that $(f_{i})_{i\in I}$ satisfies
	the open set condition.  Let $x_{0}\in J$ and let $\epsilon>0$.
	Let $\vec{n}\in\N_{0}^{s}$ and set 
	$n:=|\vec{n}|:=\sum_{i=1}^{s}n_{i}$.
	Then there exists a constant $K>0$ such that for every $k\in\N$
	there exist points $a_{k}\in B(x_{0},\epsilon)\cap J\setminus\left\{ 
	x_{0}\right\} $
	and $b_{k}\in B(x_{0},\epsilon)\setminus\left\{ x_{0}\right\} $ with
	$u_{\vec{0}}(a_{k},b_{k})\neq0$ such that for $\vec{0}\le\vec{q}\le\vec{n}$,
	\[
	K^{-1}k^{\sum_{i=1}^{s}q_{i}\left(n+1\right)^{i-1}}
	\le\frac{u_{\vec{q}}(a_{k},b_{k})}{u_{\vec{0}}(a_{k},b_{k})}\le 
	Kk^{\sum_{i=1}^{s}q_{i}\left(n+1\right)^{i-1}}.
	\]
\end{lem}
\begin{proof}
Let $\epsilon>0$. There exists $\omega\in \Sigma$ such that $\pi(\omega)=x_0$. 
Since $(f_i)_{i\in I}$ is expanding and the open set  $O$ is 
bounded, there exists $r\in \N$ such that, with $\tau:=\omega_1\dots \omega_r 
\in I^r$, we have $\diam(f_\tau^{-1}(\overline{O}))<\epsilon$. By Lemma 
\ref{lem:L is perfect}	there exist $a,b\in (J\cap O)\setminus G(x_0)$, $a\neq 
b$, such that $u_\vec{0}(a,b)=T_{\vec{p}}(a)- T_{\vec{p}}(b)\neq 0$.  For each 
$k\in\N$ and $1\le i\le s$ we set 
$m_{i}(k):=k^{\left(n+1\right)^{i-1}}$ and  
$\xi_{k}:=(1^{m_{1}(k)},2^{m_{2}(k)},\dots,s^{m_{s}(k)})\in 
I^{\sum_{i=1}^{s}m_{i}(k)}$
where $u^{m}:=(u,u,\dots,u)\in I^{m}$ for $u\in\left\{ 1,\dots,s+1\right\} $. 
Then we define $\tilde{a_k}:=f_{\xi_k}^{-1}(a)$, 
$\tilde{b_k}:=f_{\xi_k}^{-1}(b)$ as well as $a_k:=f_{\tau}^{-1}(\tilde{a_k})$, 
$b_k:=f_{\tau}^{-1}(\tilde{b_k})$. By Lemma \ref{lem:functional equation via 
cocycle} and the fact that 
\[
u_{\vec{0}}(a_k,b_k)=p_\tau \cdot u_{\vec{0}}(\tilde{a}_k,\tilde{b}_k)
\] 
it follows that 
\begin{equation}
\label{eq:u0akbk1}
\left(u_{\vec{0}}(a_{k},b_{k})\right)^{-1}U(a_{k},b_{k})=\left(u_{\vec{0}}(\tilde{a}_{k},\tilde{b}_{k})\right)^{-1}A(\overline{\tau},r)U(\tilde{a}_{k},\tilde{b}_{k}).
\end{equation} 
Similarly, we obtain that
\begin{equation}
\label{eq:u0akbk2}
\left(u_{\vec{0}}(\tilde{a}_{k},\tilde{b}_{k})\right)^{-1}U(\tilde{a}_{k},\tilde{b}_{k})=\left(u_{\vec{0}}(a,b)\right)^{-1}A\left(\overline{\xi_{k}},\sum_{i=1}^{s}m_{i}(k)\right)U(a,b).
\end{equation} 

By combining the previous two equalities 
(\ref{eq:u0akbk1}) and (\ref{eq:u0akbk2}) we have
\[
\left(u_{\vec{0}}(a_{k},b_{k})\right)^{-1}U(a_{k},b_{k})=\left(u_{\vec{0}}(a,b)\right)^{-1}A(\overline{\tau},r)
A\left(\overline{\xi_{k}},\sum_{i=1}^{s}m_{i}(k)\right)U(a,b).
\]

Since $\xi_k \in \{1,\dots ,s\}^*$ and $u_{\vec{0}}(a,b)\neq 0$ it follows from 
Lemmas 
\ref{lem:idnm} and   \ref{lem:growth of 
cocycle}  that for $\vec{q}\le \vec{n}$,
 \begin{eqnarray*}
	\left(A\left(\overline{\xi_{k}},\sum_{i=1}^{s}m_{i}(k)\right)U(a,b)\right)_{\vec{q}}
	& \asymp & \prod_{i=1}^{s}(m_{i}(k))^{q_{i}}\asymp 
	k^{\sum_{i=1}^{s}q_{i}\left(n+1\right)^{i-1}} 
	\mbox{ as }k\rightarrow \infty ,
\end{eqnarray*}
where for any two non-negative functions $\phi _{1}(k)$ and $\phi _{2}(k)$ 
of $k\in \Bbb{N}$, we write  
$\phi _{1}(k)\asymp \phi _{2}(k)$ as $k\rightarrow \infty $ 
if there exists a constant $D>1$ such that 
$D^{-1}\phi _{2}(k)\leq \phi _{1}(k)\leq D\phi _{2}(k)$ for every 
$k\in  \Bbb{N}.$ 
From this and Lemma \ref{lem:idnm} we conclude that as $k\rightarrow \infty$,

$$\left(A(\overline{\tau },r)A\left(\overline{\xi _{k}},\sum 
_{i=1}^{s}m_{i}(k)\right)U(a,b)\right)_{\vec{q}}
=\sum _{\vec{r}\leq \vec{q}}
A(\overline{\tau },r)_{\vec{q},\vec{r}}\left(A\left(\overline{\xi _{k}},\sum 
_{i=1}^{s}m_{i}(k)\right)U(a,b)\right)_{\vec{r}}\asymp 
k^{\sum_{i=1}^{s}q_{i}\left(n+1\right)^{i-1}}. $$ 
The proof is complete.
\end{proof}

We are now in the position to derive the following key lemma. Since the proof 
follows closely the arguments given in \cite[Lemma 5.2]{JS2016} we comment only 
on the necessary modifications.   Note that in 
\cite[Lemma 5.2]{JS2016}
the Julia set $J_{\omega}$ should be replaced by $J(G)$. 
\begin{lem}
\label{lem:good points are dense} Suppose that $(f_{i})_{i\in I}$
satisfies the open set condition.  Let  
$C=\sum_{\vec{n} }\beta_{\vec{n}}C_{\vec{n}}\in\T$
be non-trivial. Let $j(k)\rightarrow\infty$ be a sequence of positive
integers. Let $\omega\in\Sigma$. For every $x\in J$ and for any non-empty 
neighbourhood $V$ of $x$
in $\R$ there exist $a,b\in V\cap O$ with
$a\neq b$ such that 
\[
\eta:=\limsup_{k\rightarrow\infty}\left|\sum_{\vec{m}}\sum_{\vec{n}}\beta_{\vec{n}}A(\omega,j(k))_{\vec{n},\vec{m}}u_{\vec{m}}(a,b)\right|\in(0,\infty].
\]
\end{lem}
\begin{proof}
There exists $\vec{n}_{\max}\in \N_0^s$ such that 
$\beta_{\vec{n}}=0$ for all 
$\vec{n}\ge \vec{n}_{\max}$. By Lemma~\ref{lem:idnm}, it is easy to see that the matrix 
$\left(A(\omega,j(k))_{\vec{n},\vec{m}}\right)_{\vec{n}\le\vec{n}_{\max},\vec{m}\le\vec{n}_{\max}}$
is invertible.  Since 
$\left(\beta_{\vec{n}}\right)_{\vec{n}\le\vec{n}_{\max}}\neq0$
we conclude that, for all $k\in\N$, 
\[
\lambda(k):=\left(\lambda_{\vec{m}}(k)\right)_{\vec{m}\le\vec{n}_{\max}}:=\left(\sum_{\vec{n}\le\vec{n}_{\max}}\beta_{\vec{n}}
A(\omega,j(k))_{\vec{n},\vec{m}}\right)_{\vec{m}\le\vec{n}_{\max}}\neq 0.
\]
Let $\epsilon >0$ and suppose by way of contradiction that $\eta=0$ for all 
$a,b\in 
B(x_{0},\epsilon)\setminus 
\{ x_{0}\} $ 
with $a\neq b.$ Proceeding exactly as in the proof of \cite[Lemma 5.2]{JS2016} 
one defines 
$\lambda:=\left(\lambda_{\vec{m}}\right)_{\vec{m}\le\vec{n}_{\max}}$ as a 
limit point of the sequence $(\lambda(k)/\Vert \lambda(k)\Vert)_{k\ge 1}$ and 
observes that   $\Vert\lambda\Vert=1$ and  
\begin{equation} \label{eq:contradiction}
\sum_{\vec{m}\le\vec{n}_{\max}}\lambda_{\vec{m}}u_{\vec{m}}(a,b)=0,  
\mbox{\ for all }a,b\in B(x_{0},\epsilon )\setminus \{ x_{0}\}. 
\end{equation}
To derive the desired contradiction one verifies that there exist 
$(a_{\vec{r}})_{\vec{r} \le \vec{n}_{\max} },(b_{\vec{r}})_{\vec{r} \le 
\vec{n}_{\max} }$ with $a_{\vec{r}}, b_{\vec{r}} \in B(x_0,\epsilon)\setminus 
\{x_0\}$, for every $ \vec{r} \le 
\vec{n}_{\max} $ such that the matrix 
\[
\left(u_{\vec{q}}(a_{\vec{r}},b_{\vec{r}})\right)_{\substack{\vec{r}\le\vec{n}_{\max}\\
		\vec{q}\le\vec{n}_{\max}
	}
}
\]
is invertible. Hence, it follows from (\ref{eq:contradiction})  that 
$\lambda=0$ 
contradicting $ \Vert \lambda \Vert=1$.  The existence of the vectors 
$(a_{\vec{r}})_{\vec{r} \le 
\vec{n}_{\max} }$  and $(b_{\vec{r}})_{\vec{r} \le \vec{n}_{\max} }$ as stated 
above is 
demonstrated in \cite[Proposition 4.11]{JS2016}. The key is to combine Lemma 
\ref{lem:growth of solutions} with the idea of the Vandermonde determinant (see 
also \cite[Lemma 4.10]{JS2016}). 
\end{proof}

\subsection{Dynamical characterisation of pointwise H\"older exponents \label{subsec:Dynamical-characterisation-of}}

In this section we provide a dynamical characterisation of the pointwise
H\"older exponent of non-trivial elements of $\T$. The\emph{ pointwise
H\"older exponent} of $C$ at $x\in\R$ is defined as 
\[
\Hol(C,x):=\sup\left\{ \alpha>0\mid\limsup_{y\rightarrow 
x}\frac{\left|C(y)-C(x)\right|}{\left|y-x\right|^{\alpha}}<\infty\right\}\in 
[0,\infty] .
\]
We remark that it is possible that $\Hol(C,x)>1$. In this case,  $C$ is 
differentiable at $x$ and $C'(x)=0$.  In fact, in many examples there exists a 
set $A\subset J$ of positive Hausdorff dimension such that for every $x\in A$ 
we have 
$\Hol(C,x)>1$. We 
also note that if  $\Hol(C,x)<1$ then $C$ is not differentiable at $x$. There 
exist examples for which there exist  sets $A,B\subset J$ of positive Hausdorff 
dimension such that for each $x\in A$ we have $\Hol(C,x)>1$ whereas for each 
$x\in B$ we have $\Hol(C,x)<1$. For example, let $s=1$ and suppose that 
$(f_1,f_2)$ satisfies the open set condition. For  $p_1>0$ sufficiently close 
to zero we have that $\alpha_-<1$  and $\alpha_+>1$ by 
(\ref{defalphaplusminus}).  
Then the existence of the sets $A,B$ as stated above follows from Theorem  
\ref{thm:multifrac}.

By  \cite[Lemma 5.1]{JS13b}) we have 
for every $x\in\R$, 
\begin{equation} 
\Hol(C,x)=\liminf_{r\rightarrow0}\frac{\log\sup_{y\in B(x,r)}\left|C(y)-C(x)\right|.}{\log r}.\label{eq:hoelder exponent as liminf}
\end{equation}
\begin{rem}
	In the proof of  \cite[Lemma 5.1]{JS13b}) we have demonstrated that 
\begin{equation*} \label{eq:hoelder exponent as liminf variant}
	\Hol(C,x)=\liminf_{y\rightarrow 
	x}\frac{\log\left|C(x)-C(y)\right|}{\log\left|x-y\right|}.
\end{equation*}
	From this, it is straightforward to derive (\ref{eq:hoelder exponent as 
	liminf}). For the convenience of the reader, we give the details. For a 
	bounded function $h:V \rightarrow \R$ defined on some domain $V\subset \R$ 
	and $x\in V$ we will show that 
	\begin{equation} \label{eq:liminf1}
\liminf_{y\rightarrow 
	x}\frac{\log\left|h(x)-h(y)\right|}{\log\left|x-y\right|}=
\liminf_{r\rightarrow0}\frac{\log\sup_{y\in	B(x,r)}\left|h(y)-h(x)\right|}{\log 
r}.	
\end{equation}
Denote the left-hand side of (\ref{eq:liminf1}) by $H$, the right-hand side 
of (\ref{eq:liminf1}) by $H'$. We will show that $H=H'$. First observe that if 
$h$ is not continuous at $x$ then $H=H'=0$. Now, assume that $h$ is continuous 
at $x$.
Let $(y_{n})$ be a 
sequence with $y_{n}\neq x$, and $y_{n}\rightarrow 
x$
as $n\rightarrow\infty$ such that 
\[
\lim_{n\rightarrow 
\infty}\frac{\log\left|h(y_{n})-h(x)\right|}{\log\left|y_{n}-x\right|}=H.
\]
We may assume that $r_{n}:=\left|y_{n}-x\right|<1$
for all $n\ge1$.  It follows that 
\[
H'\le \liminf_{n\rightarrow 
\infty}\frac{\log\sup_{y\in B(x,r_{n})}\left|h(y)-h(x)\right|}{\log 
r_{n}}\le\liminf_{n\rightarrow 
\infty}\frac{\log\left|h(y_{n})-h(x)\right|}{\log\left|y_{n}-x\right|}=H.
\]
To prove the reverse inequality, let $(r_{n})$
be a sequence with $r_{n}>0$ and $r_{n}\rightarrow0$, as $n\rightarrow\infty$
such that 
\[
\lim_{n\rightarrow\infty}\frac{\log\sup_{y\in B(x,r_{n})}
	\left|h(y)-h(x)\right|}{\log 
r_{n}}= H'.
\]
Then there exists a sequence $(y_{n})$ with $y_{n}\in B(x,r_{n})$
such that 
\[
\lim_{n\rightarrow\infty}\frac{\log\left|h(y_{n})-h(x)\right|}{\log r_{n}}= H'.
\]
Since $h$ is continuous at $x$ we may assume that 
$\log\left|h(y_{n})-h(x)\right|<0$
for all $n\ge1$. Hence, we have 
\[
H\le \liminf_{n\rightarrow \infty} 
\frac{\log\left|h(y_{n})-h(x)\right|}{\log\left|y_{n}-x\right|}
	\le \liminf_{n\rightarrow \infty}\frac{\log\left|h(y_{n})-h(x)\right|}{\log
 r_{n}}= H'.
\]
This completes the proof of $H=H'$. 
\end{rem}
We proceed with  an upper bound for the pointwise
H\"older exponent.	
\begin{prop}
\label{prop:upperbound pointwise hoelder}Suppose that $(f_{i})_{i\in I}$
satisfies the open set condition.  Let $C=\sum_{\vec{n}}\beta_{\vec{n}}C_{\vec{n}}\in\T$
be non-trivial. For every $x\in J$ we have 
\[
\Hol(C,x)\le\inf_{\omega\in\pi^{-1}(x)}\liminf_{n\rightarrow\infty}\frac{S_{n}\psi(\omega)}{S_{n}\varphi(\omega)}.
\]
\end{prop}

\begin{proof}
Let $x\in J$ and $\omega\in\pi^{-1}(x)$. Since $J$ is compact,
there exists a sequence $(j_{k})$ tending to infinity and $x_{0}\in J$
such that 
\[
\alpha:=\liminf_{n\rightarrow\infty}\frac{S_{n}\psi(\omega)}{S_{n}\varphi(\omega)}=\lim_{k\rightarrow\infty}\frac{S_{j_{k}}\psi(\omega)}{S_{j_{k}}\varphi(\omega)}\quad\text{and}\quad\lim_{k\rightarrow\infty}f_{\omega_{|j_{k}}}(x)=x_{0}\in J.
\]
By Lemma \ref{lem:good points are dense} we may assume that there exist 
$\epsilon>0$,
$\eta_{0}>0$ and points $a,b\in B(x_{0},\epsilon)\cap O$ with
$a\neq b$ such that for all $k$ sufficiently large, 
\[
\left|\sum_{\vec{m}}\sum_{\vec{n}}\beta_{\vec{n}}A(\omega,j(k))_{\vec{n},\vec{m}}u_{\vec{m}}(a,b)\right|\ge\eta_{0}>0.
\]
Define $y_{k}:=(f_{\omega_{|j_{k}}})^{-1}(a)$ and $z_{k}:=(f_{\omega_{|j_{k}}})^{-1}(b)$.
By Lemma \ref{lem:functional equation via cocycle} we have 
\begin{align*}
C(y_{k})-C(z_{k}) & =\sum_{\vec{n}}\beta_{\vec{n}}\left(U(y_{k},z_{k})\right)_{\vec{n}}=p_{\omega_{|j(k)}}\sum_{\vec{n}}\beta_{\vec{n}}\left(A(\omega,j(k))U(a,b)\right)_{\vec{n}}.
\end{align*}
Hence, we can estimate 
\[
\liminf_{k\rightarrow\infty}\frac{\log|C(y_{k})-C(z_{k})|}{\log|y_{k}-z_{k}|}\le\liminf_{k\rightarrow\infty}\frac{S_{j(k)}\psi(\omega)+\log\eta_{0}}{S_{j(k)}\varphi(\omega)}=\alpha.
\]
Finally, proceeding exactly as in \cite[Proof of Lemma 5.3 (5.4), (5.5) and 
(5.6)]{JS2016},
the statement of the lemma follows.
\end{proof}
Since $\Hol(C,x)<\infty$ by Proposition \ref{prop:upperbound pointwise hoelder},
we can conclude the following.
\begin{cor}
Suppose that $(f_{i})_{i\in I}$ satisfies the open set condition.
Let $C\in\T$ be non-trivial.
Then $C$ is not locally constant at any point of $J$ and thus $\T$
is the direct sum of vector spaces $\oplus_{\vec{n}\in\N_{0}^{s}}\R C_{\vec{n}}$.
\end{cor}

Next we provide lower bounds for the pointwise H\"older exponent.
In the following proposition we define for $\omega\in\Sigma$ and
$x=\pi(\omega)$ the sequences
\[
\delta_{n}:=\delta_{n}(\omega):=d\left(f_{\omega_{|n}}(x),\partial O\right),\quad\text{and }s_{n}:=s_{n}(\omega):=\delta_{n}\cdot\left|f_{\omega_{|n}}'(x)\right|^{-1}\cdot D^{-1},
\]
where $\partial O$ refers to the boundary of the open interval $O$
and $D=D(O)$ refers to the bounded distortion constant in Lemma \ref{lem:bounded distortion}.
\begin{prop}
\label{prop:lower bound} Suppose that $(f_{i})_{i\in I}$ satisfies
the open set condition.  Let $C\in\T$ be non-trivial. Let $\omega\in\Sigma$
be a sequence which is not eventually constant and let $x=\pi(\omega)\in J$.
Then we have 
\[
\Hol(C,x)\cdot\left(1+\frac{\liminf_{n\rightarrow\infty}n^{-1}\log\delta_{n}}{\max\varphi}\right)\ge\liminf_{n\rightarrow\infty}\frac{S_{n}\psi(\omega)}{S_{n}\varphi(\omega)}.
\]
 
\end{prop}

\begin{proof}
By \eqref{eq:hoelder exponent as liminf} there exists a sequence
$(r_{k})$ tending to zero such that 
\[
\Hol(C,x)=\lim_{k\rightarrow\infty}\frac{\log\sup_{y\in B(x,r_{k})}\left|C(x)-C(y)\right|}{\log r_{k}}.
\]
Let $n_{k}:=\max\left\{ n\ge1\mid s_{n}>r_{k}\right\} $. Since $\omega$
is not eventually constant, we have $s_{n}>0$ for each $n\in \Bbb{N}$. Hence, 
$(n_{k})\rightarrow\infty$,
as $k\rightarrow\infty$. By the definition of $n_{k}$ we have 
\[
B(x,r_{k})\subset B_{k}:=\left(f_{\omega_{|n_{k}}}\right)^{-1}\left(B\left(f_{\omega_{|n_{k}}}(x),\delta_{n_{k}}\right)\right),\quad k\in\N.
\]
Hence, if $r_{k}<1$ then 
\[
\frac{\log\sup_{y\in B(x,r_{k})}\left|C(x)-C(y)\right|}{\log r_{k}}\ge\frac{\log\sup_{y\in B_{k}}\left|C(x)-C(y)\right|}{\log r_{k}}.
\]
Also, by the definition of $n_{k}$, we have $s_{n_{k}+1}\le r_{k}$.
Therefore, we have 
\[
\frac{\log\sup_{y\in B_{k}}\left|C(x)-C(y)\right|}{\log r_{k}}\ge\frac{\log\sup_{y\in B_{k}}\left|C(x)-C(y)\right|}{\log s_{n_{k}+1}}=\frac{\log\sup_{y\in B_{k}}\left|C(x)-C(y)\right|}{-\log D+S_{n_{k}+1}\varphi(\omega)+\log\delta_{n_{k}+1}}.
\]
By Lemmas \ref{lem:functional equation via cocycle} and  \ref{lem:growth of 
cocycle}
there exist constants $K'$ and $q$, which are independent of $k$,
such that 
\begin{align*}
\sup_{y\in B_{k}}\left|C(x)-C(y)\right| & \le K'p_{\omega_{|n_{k}}}n_{k}^{q}.
\end{align*}
We conclude that 
\begin{align*}
\frac{\log\sup_{y\in B_{k}}\left|C(x)-C(y)\right|}{-\log D+S_{n_{k}+1}\varphi(\omega)+\log\delta_{n_{k}+1}} & \ge\frac{S_{n_{k}}\psi(\omega)+q\log(n_{k})+\log K'}{-\log D+S_{n_{k}+1}\varphi(\omega)+\log\delta_{n_{k}+1}}.\\
 & =\frac{S_{n_{k}}\psi(\omega)\cdot\left(1+(q\log(n_{k})+\log K')/S_{n_{k}}\psi(\omega)\right)}{S_{n_{k}+1}\varphi(\omega)\cdot\left(1+\left(-\log D+\log\delta_{n_{k}+1}\right)/S_{n_{k}+1}\varphi(\omega)\right)}
\end{align*}
and the claim follows by letting $k$ tend to infinity. 
\end{proof}
The following result is the  analogue of  \cite[Lemma 5.1 and Lemma 
5.3]{JS2016}.
\begin{cor}
\label{cor:dynamical formulation of hoelder exponent}Suppose that
$(f_{i})_{i\in I}$ satisfies the separating condition. Let $C\in\T$ be non-trivial.
Then for all $\omega\in\Sigma$ and $x=\pi(\omega)$, we have 
\[
\liminf_{n\rightarrow\infty}\frac{S_{n}\psi(\omega)}{S_{n}\varphi(\omega)}=\Hol(C,x).
\]
\end{cor}

\begin{proof}
Since $\inf_{n\in\N}\delta_{n}>0$, we have $\liminf_{n\rightarrow\infty}n^{-1}\log\delta_{n}=0$.
Hence, $\liminf_{n\rightarrow\infty}S_{n}\psi(\omega)\big/S_{n}\varphi(\omega)\le\Hol(C,x)$
by Proposition \ref{prop:lower bound}. The converse inequality follows
from Proposition \ref{prop:upperbound pointwise hoelder}.
\end{proof}
Let us end this section with the following remark concerning Proposition
\ref{prop:lower bound} and Corollary \ref{cor:dynamical formulation of hoelder exponent}.
\begin{rem}
It is not difficult to find examples of systems $(f_{i})_{i\in I}$
satisfying the open set condition with limit points $x=\pi(\omega)\in J$,
$\omega\in\Sigma$, such that 
\[
\Hol(C_{0},x)<\lim_{n\rightarrow\infty}\frac{S_{n}\psi(\omega)}{S_{n}\varphi(\omega)}.
\]
We refer to \cite[Example 3.1]{JS20} for the details.
Hence, in contrast to systems satisfying the separating condition
(see Corollary \ref{cor:dynamical formulation of hoelder exponent}),
the dynamical characterization of the pointwise H\"older exponent
in terms of quotients of ergodic sums does not always hold for systems
satisfying the open set condition. However, Proposition \ref{prop:lower bound}
can be used to establish a dynamical characterization for almost every
limit point with respect to suitable reference measures. Moreover,
it turns out in Theorem \ref{thm:multifrac} below that the dimension
spectrum of the pointwise H\"older exponents coincides with the spectrum
of quotients of ergodic sums of the potentials $\varphi$ and $\psi$.
\end{rem}

\subsection{Dimension spectrum of pointwise H\"older exponents\label{subsec:Dimension-spectrum-of}}

We define 
\[
\mathcal{F}(\alpha):=\mathcal{F}_{\vec{p}}(\alpha):=\pi\left\{ \omega\in\Sigma\mid\lim_{n\rightarrow\infty}\frac{S_{n}\psi(\omega)}{S_{n}\varphi(\omega)}=\alpha\right\} .
\]
Suppose that $(f_{i})_{i\in I}$ satisfies the open set condition.
It is well known that the multifractal spectrum is complete (\cite{MR1738952}),
that is, there exist $\alpha_{-},\alpha_{+}\in\R$ such that $\mathcal{F}(\alpha)\neq\emptyset$
if and only if $\alpha\in\left[\alpha_{-},\alpha_{+}\right]$. 
For every $\beta\in\R$ there exists a unique $t(\beta)\in\R$ such
that $\mathcal{P}(t(\beta)\varphi+\beta\psi)=0$, where $\mathcal{P}(u)$
refers to the topological pressure of a continuous function $u$ with
respect to the dynamical system $(\Sigma,\sigma)$ (see \cite{MR648108}).
Note that  $\varphi$ is H\"older continuous since the dynamical system is  
expanding  with $\mathcal{C}^{1+\epsilon}$ branches. Also,  $\psi$ is H\"older 
continuous as it only depends on the first coordinate. By well-known results 
from thermodynamic 
formalism for H\"older continuous potentials, it follows that the function $t$ 
is real-analytic and convex
function with $t'(\beta)=-\int\psi\,\,d\mu_{\beta}/\int\varphi\,\,d\mu_{\beta}$
where $\mu_{\beta}$ denotes the unique Gibbs probability measure
on $\Sigma$ associated with $t(\beta)\varphi+\beta\psi$. Moreover, with  
$\alpha_{-}$ and $\alpha_{+}$ given by \eqref{defalphaplusminus}, we have that 
the 
function
$t$ satisfies $t''>0$ if and only if $\alpha_{-}<\alpha_{+}$, and
have that $\alpha_{-}=\alpha_{+}$ if and only if $\delta\varphi$
and $\psi$ are cohomologous, where 
\[
\delta:=t(0)=\dim_{H}(J).
\]
Here, we say that $\delta\varphi$ and $\psi$ are cohomologous if
there exists a continuous function $\kappa:\Sigma\rightarrow\R$ such
that $\delta\varphi=\psi+\kappa-\kappa\circ\sigma$. Note that we
have $-t'(\R)=(\alpha_{-},\alpha_{+})$ if $\alpha_{-}<\alpha_{+}$,
and $-t'(\R)=\{\alpha_{-}\}$, otherwise. We  define the level
sets 
\[
\mathcal{F}^{\#}(\alpha):=\begin{cases}
\pi\left\{ \omega\in\Sigma\mid\limsup_{n\rightarrow\infty}\frac{S_{n}\psi(\omega)}{S_{n}\varphi(\omega)}\ge\alpha\right\} , & \alpha\ge\alpha_{0}\\
\pi\left\{ \omega\in\Sigma\mid\liminf_{n\rightarrow\infty}\frac{S_{n}\psi(\omega)}{S_{n}\varphi(\omega)}\le\alpha\right\} , & \alpha<\alpha_{0},
\end{cases}
\]
where we have set $\alpha_{0}:=\int\psi\,\,d\mu_{0}/\int\varphi\,\,d\mu_{0}.$ We
denote the convex conjugate of $t$ (\cite{rockafellar-convexanalysisMR0274683})
by
\[
t^{*}(u):=\sup\left\{ \beta u-t(\beta)\mid\beta\in\R\right\} \in\R\cup\{+\infty\}.
\]
It is well-known (see e.g. \cite{pesindimensiontheoryMR1489237,MR1738952})
that for $\alpha\in[\alpha_{-},\alpha_{+}]$, 
\begin{equation}
\dim_{H}\left(\mathcal{F}(\alpha)\right)=\dim_{H}\left(\mathcal{F}^{\#}(\alpha)\right)
=-t^{*}(-\alpha)\ge0,\label{eq:multifractal formalism}
\end{equation}
$-t^{*}(-\alpha)>0$ for $\alpha\in(\alpha_{-},\alpha_{+})$ if $\alpha 
_{-}<\alpha _{+}$, and that 
$\mathcal{F}(\alpha)=\mathcal{F}^{\#}(\alpha)=\emptyset$
for $\alpha\notin[\alpha_{-},\alpha_{+}]$. To prove this, it is shown
that if $\beta\in\R$ and $\alpha=-t'(\beta)$, then for the corresponding
Gibbs measure $\mu_{\beta}$ we have $\mu_{\beta}\circ\pi^{-1}(\mathcal{F}(\alpha))=1$
and 
\begin{equation}
\dim_{H}\left(\mathcal{F}(\alpha)\right)=\dim_{H}(\mu_{\beta}\circ\pi^{-1})>0.\label{eq:multifractal formalism measure theoretic}
\end{equation}
We refer to \cite{JS13b,JS2016} for a closely related framework for
random complex dynamical systems. See in particular \cite[Remark 3.14, 
Proposition 4.4, Theorem 5.3]{JS13b}.
If $\alpha_{-}=\alpha_{+}$ then $\mathcal{F}(\alpha_{-})=J$ and
for every $\beta\in\R$, 
\[
\dim_{H}(\mathcal{F}(\alpha_{-}))=\dim_{H}(\mu_{\beta}\circ\pi^{-1})=\dim_{H}(J)=t(0)=\delta.
\]
\begin{cor}
\label{cor:lowerbound-almosteverywhere}Suppose that $(f_{i})_{i\in I}$
satisfies the open set condition.   Let $C\in \T $
be non-trivial. Then for all $\alpha\in[\alpha_{-},\alpha_{+}]$ we
have 
\[
\dim_{H}\left\{ x\in J\mid\Hol(C,x)=\alpha\right\} \ge-t^{*}(-\alpha).
\]
\end{cor}

\begin{proof}
It is well known that, for each $\alpha\in[\alpha_{-},\alpha_{+}]$,
there exists an ergodic Borel probability measure $\mu$ with $\int\psi\,\,d\mu/\int\varphi\,\,d\mu=\alpha$
and $\dim_{H}(\mu\circ\pi^{-1})=-t^{*}(-\alpha)\ge0$ (see e.g. \cite{pesindimensiontheoryMR1489237,MR1738952}).
Suppose that $-t^{*}(-\alpha)>0$; otherwise there is nothing to prove.
Following \cite{MR1479016} we have for $\mu$-a.e. $\omega\in\Sigma$
and $x=\pi(\omega)$,  $\lim_{n\rightarrow\infty}n^{-1}\log\delta_{n}=0$,
where $\delta_{n}$ is defined prior to Proposition \ref{prop:lower bound}.
Hence, for $\mu$-a.e. $\omega\in\Sigma$ and $x=\pi(\omega)$, $\Hol(C,x)\ge\alpha$
by Proposition \ref{prop:lower bound}. Moreover, by Proposition \ref{prop:upperbound pointwise hoelder}
we have $\Hol(C,x)\le\alpha$ $\mu$-a.e. We conclude that $\mu\circ\pi^{-1}\left(\left\{ x\in J\mid\Hol(C,x)=\alpha\right\} \right)=1$.
Thus, $\dim_{H}\left\{ x\in J\mid\Hol(C,x)=\alpha\right\} \ge\dim_{H}(\mu\circ\pi^{-1})=-t^{*}(-\alpha)$.
\end{proof}
The following proposition is an extension of results of  Allaart
(\cite{Allaart17}) for self-similar measures.
\begin{prop}
\label{prop:lower half of spectrum upper bound}Suppose that $(f_{i})_{i\in I}$
satisfies the open set condition. For every 
$\alpha\in\left[\alpha_{-},\alpha_{0}\right]$
and every $C\in\T$ we have 
\[
\dim_{H}\left\{ x\in J\mid\Hol(C,x)=\alpha\right\} \le-t^{*}(-\alpha).
\]
Further, if $\alpha<\alpha_{-}$ then $\left\{ x\in J\mid\Hol(C,x)=\alpha\right\} =\emptyset$. 
\end{prop}

\begin{proof}
We use $\diam(A):=\sup\{d(x,y)\mid x,y\in A\} $ to denote the diameter of a set 
$A\subset \R$. We first observe that by the H\"older continuity of $\varphi$, 
there exists a constant $D\ge 1$ such that 
\[
\diam(\pi([\gamma i]))\ge D^{-1} \diam(\pi([\gamma ]))
\]
for all $i\in I$ and for all $\gamma \in I^*$. 
Let  $c\in J$ and $x=\pi(\omega)\in J$ with $\Hol(C,x)=\alpha$. By 
\eqref{eq:hoelder exponent as liminf}
there exists $r_{k}\rightarrow0$ such that 
\[
\alpha=\lim_{k\rightarrow\infty}\log\sup_{y\in 
B(x,r_{k})}\left|C(y)-C(x)\right|/\log r_{k}.
\]
Define $n_{k}:=\min\left\{ n\ge1\mid\diam(\pi([\omega_{|n}]))<D^2 r_{k}\right\} 
$,
$k\ge1$. If 
\[
(\omega_{1},\dots,\omega_{n_{k}})=(\tau j(s+1)^{\ell_k}),
\]
for some $\tau\in I^{*}$, $j\le s$ and $\ell_k\ge1$ then we define 
$\nu_{k},\nu_{k}'\in I^{*}$
 by 
\[
\nu_{k}:=(\tau 
j(s+1)^{\ell_k-1})=(\omega_{1},\dots,\omega_{n_{k}-1}),\quad\nu_{k}':=(\tau(j+1)1^{\ell'_k}),
\]
where $\ell'_k\ge1$ is given by 
\[
\ell'_k:=\max \{\ell \ge 1\mid \diam(\pi([\tau (j+1) 1^\ell]))\ge r_k \} .
\]
Similarly, if $(\omega_{1},\dots,\omega_{n_{k}})=(\tau(j+1)1^{\ell_k})$,
then we define $\nu_{k}:=(\omega_{1},\dots,\omega_{n_{k}-1})$ and
$\nu'_{k}:=(\tau j(s+1)^{\ell'_k})$, where 
$\ell'_k:=\max \{\ell \ge 1\mid \diam(\pi([\tau j (s+1)^\ell]))\ge r_k \}$. 
Note that the sequence
$\ell_k'$ 
is well-defined. If $\omega_{n_{k}}\notin\{1,s+1\}$
then we define $\nu_{k}:=\nu_{k}':=(\omega_{1},\dots,\omega_{n_{k}-1})$.
Let $n_{k}':=|\nu_{k}'|$. 

It is important to note that, by the definition of $\nu_{k}$ and
$\nu_{k}'$, we have as $k\rightarrow\infty$, 
\begin{equation}
\diam\left(\pi\left[\nu_{k}\right]\right)\asymp\diam\left(\pi\left[\nu_{k}'\right]\right)\asymp
 r_{k}.\label{eq:comparable}
\end{equation}
Moreover, there exists a  constant $E\ge 0$ such that 
\begin{equation}
\ell_{k}'\cdot\varphi(\overline{1})-E\le\ell_{k}\cdot\varphi(\overline{s+1})\le
\ell_{k}'\cdot\varphi(\overline{1})+E.\label{eq:comparability
 lk}
\end{equation}
To prove \eqref{eq:comparability lk}
suppose that $(\omega_{1},\dots,\omega_{n_{k}})=(\tau j(s+1)^{\ell_{k}})$.
The other case $(\omega_{1},\dots,\omega_{n_{k}})=(\tau(j+1)1^{\ell_{k}}),$
can be handled analogously. By the H\"older continuity of $\varphi$ we have, as 
$k\rightarrow\infty$, 
\[
\diam\left(\pi\left[\nu_{k}\right]\right)=\diam\left(\pi\left[(\tau 
j(s+1)^{\ell_{k}-1})\right]\right)\asymp\diam\left(\pi\left[\tau\right]\right)
\e^{S_{\ell_{k}}\varphi(\overline{(s+1)})}
\]
\[
\diam\left(\pi\left[\nu_{k}'\right]\right)=\diam\left(\pi\left[\tau(j+1)1^{\ell'_{k}}\right]\right)\asymp\diam\left(\pi\left[\tau\right]\right)\e^{S_{\ell'_{k}}\varphi(\overline{1})},
\]
which proves \eqref{eq:comparability lk}.

We will show that for  $k\ge1$ 
\[
B(x,r_{k})\cap J \subset\pi([\nu_{k}])\cup\pi([\nu'_{k}]).
\]

First suppose that $(\omega_{1},\dots,\omega_{n_{k}})=(\tau j(s+1)^{\ell_{k}})$.
Then $\pi\left[\nu_{k}\right]=\pi\left[(\tau 
j(s+1)^{\ell_{k}-1})\right]\supset\pi\left[(\tau j(s+1)^{\ell_{k}-1}1)\right]$.
Since $x\in\pi\left[(\tau j(s+1)^{\ell_{k}})\right]$ we have 
$x\ge\max\pi\left[(\tau j(s+1)^{\ell_{k}-1}1)\right]$.
By the definition of  $n_{k}$ we have 
\[
\diam\left(\pi\left[(\tau j(s+1)^{\ell_{k}-1}1)\right]\right)\ge 
D^{-1}\diam\left(\pi\left[(\tau 
j(s+1)^{\ell_{k}-1})\right]\right)=D^{-1}\diam\left(\pi\left[(\omega_{1},\dots,
\omega_{n_{k}-1})\right]\right)\ge
D r_{k}.
\]
Hence, 
\[
\pi\left[\nu_{k}\right]\supset\left[x,x-r_{k}\right]\cap J.
\]
Further, by the definition of $\ell'_{k}$ we have 
\[
\diam\left(\pi\left[\nu_{k}'\right]\right)=\diam\left(\pi\left[\tau(j+1)1^{\ell'_{k}}\right]\right)\ge
 r_{k},
\]
so 
$\left[x,x+r_{k}\right]\cap J 
\subset\pi\left[\nu_{k}\right]\cup\pi\left[\nu_{k}'\right]$.
This proves that 
$B(x,r_{k})\cap J \subset\pi([\nu_{k}])\cup\pi([\nu'_{k}])$.

%
%
%
%
Let $\epsilon>0$. We will derive from our assumption $\Hol(C,x)=\alpha$
that there exists $N\ge1$ such that for all $k\ge N$, 
\begin{equation}
\frac{S_{\left|\nu_{k}\right|}\psi(\overline{\nu_{k}})}{S_{\left|\nu_{k}\right|}\varphi(\overline{\nu_{k}})}\le\alpha+\epsilon\quad\text{or}\quad\frac{S_{\left|\nu_{k}'\right|}\psi(\overline{\nu_{k}'})}{S_{\left|\nu_{k}'\right|}\varphi(\overline{\nu_{k}'})}\le\alpha+\epsilon.\label{eq:nu or nuprime is good}
\end{equation}
To prove \eqref{eq:nu or nuprime is good}, we first note that by
Lemmas \ref{lem:functional equation via cocycle} and \ref{lem:growth of 
cocycle}, there exist constants $K'\ge 1$ and $q\in \N_0$ such that for every 
$\nu\in I^{*}$, 
\[
\sup_{y_{1},y_{2}\in\pi([\nu])}\left|C(y_{1})-C(y_{2})\right|\le 
K'\exp(S_{|\nu|}\psi(\overline{\nu}))|\nu|^{q}.
\]
Suppose for a contradiction that \eqref{eq:nu or nuprime is good} does not hold.
Then, by passing to a subsequence of $(n_{k})$ we may assume that
for all $k$ and for all $\nu\in\{\nu_{k},\nu'_{k}\}$ we have $S_{|\nu|}\psi(\overline{\nu})\big/S_{|\nu|}\varphi(\overline{\nu})\ge\alpha+\epsilon$,
and hence, by enlarging the constant $K'$ if necessary, we have 
\[
\sup_{y_{1},y_{2}\in\pi([\nu])}\left|C(y_{1})-C(y_{2})\right|\le K'\exp(S_{|\nu|}\psi(\overline{\nu}))|\nu|^{q}\le K'r_{k}^{\alpha+\epsilon}|\nu|^{q}.
\]
Since $B(x,r_{k})\cap J\subset\pi([\nu_{k}])\cup\pi([\nu'_{k}])$, we
conclude that 
\[
\lim_{k\rightarrow\infty}\log\sup_{y\in B(x,r_{k})}\left|C(y)-C(x)\right|/\log r_{k}\ge\alpha+\epsilon.
\]
This contradiction proves \eqref{eq:nu or nuprime is good}. 
Let $\beta>0$, $\eta>0$ and let $u=t(\beta)+\beta(\alpha+\epsilon)+\eta$. Note 
that by \eqref{eq:multifractal formalism} we have $u>t(\beta)+\beta(\alpha)\ge 
0$. We define
\[
\mathcal{C}_{\alpha+\epsilon}:=\left\{ \tau\in I^{*}\mid\frac{S_{\left|\tau\right|}\psi(\overline{\tau})}{S_{\left|\tau\right|}\varphi(\overline{\tau})}\le\alpha+\epsilon\right\} .
\]
We obtain a covering $\mathcal{C}$ of $\left\{x\in J \mid 
\Hol(C,x)=\alpha\right\} $
by images of cylinders of sufficiently small diameters as follows. For each $x\in J$ with
$\Hol(C,x)=\alpha$ we define  the sequence $(n_{k})$ and the sequences of 
cylinders $(\nu_k)$ and $(\nu_k')$ as above. This means in particular that 
(\ref{eq:nu or nuprime is good}) holds and   $x\in \pi([\nu_k])$ for every 
$k\ge 1$. We 
then pick an image of a cylinder $\nu(x):=\pi([\nu_k])$ of sufficiently small diameter. This 
defines a covering of   $\left\{x\in J \mid \Hol(C,x)=\alpha\right\} $ given by
\[
\mathcal{C}:=\left\{ \nu(x)\mid x\in J : \Hol(C,x)=\alpha\right\} .
\]
To verify that the corresponding
sum of diameters 
$\sum_{\nu\in\mathcal{C}}\diam\left(\pi\left(\left[\nu\right]\right)\right)^{u}$
converges, we proceed as follows. If $\nu\notin\mathcal{C}_{\alpha+\epsilon}$
then, by  \eqref{eq:comparable} and  \eqref{eq:nu or nuprime is good}, we can 
replace $\nu$
by $\nu'\in\mathcal{C}_{\alpha+\epsilon}$ satisfying
$\diam\left(\pi\left(\left[\nu\right]\right)\right)\asymp
\diam\left(\pi\left(\left[\nu'\right]\right)\right).$
 This defines a map $\nu\mapsto\nu'$ from
$\mathcal{C}\setminus\mathcal{C}_{\alpha+\epsilon}$ to 
$\mathcal{C}_{\alpha+\epsilon}$.
Since the involved numbers $\ell_{k}$ and $\ell'_{k}$ in the definition
of $\nu_k$ and $\nu'_k$ satisfy \eqref{eq:comparability lk}, we have
that    every element of $\mathcal{C}_{\alpha+\epsilon}$
is taken at most a uniformly bounded number of times under the map  
$\nu\mapsto\nu'$.  Since 
\begin{equation} 
\sum_{\omega\in\mathcal{C}_{\alpha+\epsilon}}\diam(\pi([\omega]))^{u}\asymp
\sum_{\omega\in\mathcal{C}_{\alpha+\epsilon}}\e^{\left(t(\beta)+\beta(\alpha+\epsilon)+\eta\right)S_{|\omega|}\varphi(\overline{\omega})}\le\sum_{\omega\in\mathcal{C}_{\alpha+\epsilon}}\e^{\left(t(\beta)+\eta\right)S_{\left|\omega\right|}\varphi(\overline{\omega})+\beta
 S_{\left|\omega\right|}\psi(\overline{\omega})}<\infty,\label{eq:finite s dim 
hausdorff}
\end{equation}
we therefore conclude that the $u$-dimensional Hausdorff measure
of $\left\{ x\in J\mid\Hol(C,x)=\alpha\right\} $ is finite. To prove that the 
last sum in (\ref{eq:finite s dim 
	hausdorff}) is finite, first recall that by the definition 
of $t(\beta)$ we have that $\mathcal{P}(t(\beta)\varphi+\beta\psi)=0$. Further, 
since $\varphi <0$ we conclude that $c \mapsto \mathcal{P}(c\varphi+\beta\psi)$ 
is strictly decreasing. Whence, 
$a:=\mathcal{P}((t(\beta)+\eta)\varphi+\beta\psi)<0$.  By the definition of 
topological pressure this implies that there exists a constant $b$  such that 
$\sum_{n=1}^{\infty}\sum_{\omega\in 
I^{n}}e^{S_{n}\left((t(\beta)+\eta)\varphi+\beta\psi\right)(\overline{\omega})}\le
 b\sum_{n=1}^{\infty}e^{na/2}<\infty.$ To complete the proof of the 
 proposition, 
 first assume that 
$\alpha\in\left[\alpha_{-},\alpha_{0}\right]$.
Since $\epsilon$ and $\eta$ are arbitrary, it follows that 
\[
\dim_{H}\left\{ x\in J\mid\Hol(C,x)=\alpha\right\} \le\inf_{\beta>0}\left\{ t(\beta)+\beta\alpha\right\} =-t^{*}(-\alpha),
\]
where we have used $\alpha\le\alpha_{0}$ for the last equality. Finally,
if $\alpha<\alpha_{-}$ then $\mathcal{C}_{\alpha+\epsilon}=\emptyset$
if $\alpha+\epsilon<\alpha_{-}$. By the above construction of the
 covering of $\left\{ x\in J\mid\Hol(C,x)=\alpha\right\} $, it thus
follows that $\left\{ x\in J\mid\Hol(C,x)=\alpha\right\} =\emptyset$.
The proof is complete.
\end{proof}
\begin{thm}
\label{thm:multifrac} Suppose that $(f_{i})_{i\in I}$ satisfies
the open set condition.  Let $C\in\T$ be non-trivial.
Then we have for all $\alpha\in[\alpha_{-},\alpha_{+}],$
\[
\dim_{H}\left\{ x\in J\mid\Hol(C,x)=\alpha\right\} =-t^{*}(-\alpha),
\]
and for $\alpha\not\notin[\alpha_{-},\alpha_{+}]$ we have $\left\{ x\in 
J\mid\Hol(C,x)=\alpha\right\} =\emptyset$. The function $g(\alpha)= 
-t^{*}(-\alpha)$ is continuous and concave on 
$[\alpha_{-},\alpha_{+}]$. If $\alpha_{-}<\alpha_{+}$ then  $g$ is 
real-analytic and positive on $(\alpha_{-},\alpha_{+})$ and satisfies $g''<0$ 
on  $(\alpha_{-},\alpha_{+})$.
\end{thm}

\begin{proof}
By Corollary \ref{cor:lowerbound-almosteverywhere}, we have $\dim_{H}\left\{ x\in J\mid\Hol(C,x)=\alpha\right\} \ge-t^{*}(-\alpha)$
for $\alpha\in[\alpha_{-},\alpha_{+}]$. By Proposition \ref{prop:upperbound pointwise hoelder}
we have for every $\alpha\in\R$, 
\begin{equation}
\left\{ x\in J\mid\Hol(C,x)=\alpha\right\} \subset\pi\left(\left\{ \omega\in\Sigma\mid\liminf_{n\rightarrow\infty}\frac{S_{n}\psi(\omega)}{S_{n}\varphi(\omega)}\ge\alpha\right\} \right).\label{eq:proof of upper bound}
\end{equation}
We distinguish two cases. If $\alpha\ge\alpha_{0}$, then by the definition
of $\mathcal{F}^{\#}(\alpha)$ we have $\left\{ x\in J\mid\Hol(C,x)=\alpha\right\} \subset\mathcal{F}^{\#}(\alpha)$.
Hence, by (\ref{eq:multifractal formalism}), we have $\dim_{H}\left(\left\{ x\in J\mid\Hol(C,x)=\alpha\right\} \right)\le\dim_{H}(\mathcal{F}^{\#}(\alpha))=-t^{*}(-\alpha)$
for $\alpha\ge\alpha_{0}$. Also, by \eqref{eq:proof of upper bound},
we have $\left\{ x\in J\mid\Hol(C,x)=\alpha\right\} =\emptyset$ if
$\alpha>\alpha_{+}$. For $\alpha\le\alpha_{0}$ the remaining assertions
follow from Proposition \ref{prop:lower half of spectrum upper bound}. For the 
proof of the  well-known properties of $g$ we refer to 
(\cite{pesindimensiontheoryMR1489237}, \cite{MR1738952}, see also \cite{JS13b}).
The proof is complete.  
\end{proof}
\begin{rem}
In \cite{Barany-etal16} the pointwise H\"older exponent of affine
zipper curves generated by contracting affine mappings is investigated.
The multifractal dimension spectrum is obtained only for $\alpha\ge\alpha_{0}$.
In a recent paper of  Allaart (\cite{Allaart17}) the complete
multifractal spectrum of pointwise H\"older exponents for curves
associated with selfsimilar measures is obtained.
\end{rem}

\section{Global H\"older continuity\label{sec:Global-Hlder-continuity}}

In this section we investigate the global H\"older continuity of
the elements of $\T$. The first statement of the next theorem has
been obtained for the Minkowski's question mark function in \cite{MR2444218},
for distributions of conformal iterated function systems satisfying
the separating condition in \cite{MR2525939}, and for expanding circle
diffeomorphisms in \cite{MR2576266}. 
\begin{thm}
\label{thm:alpha-minus hoelder}Suppose that $(f_{i})_{i\in I}$ satisfies
the open set condition.  Then we have the following. 
\begin{enumerate}
\item $T_{\vec{p}}\in\mathcal{C}^{\alpha_{-}}(\overline{\R})$.
\item $\T\subset\bigcap_{\alpha<\alpha_{-}}\mathcal{C}^{\alpha}(\overline{\R})$. 
\end{enumerate}
\end{thm}

\begin{proof}
We only verify the desired H\"older continuity at points $x,y\in J$.
That this is sufficient can be seen as follows. If $y\in\overline{\R}\setminus J$,
then either there exists $u\in J$ between $x$ and $y$ and the desired
H\"older continuity follows from the triangle inequality, or we have
$T_{\vec{p}}(x)=T_{\vec{p}}(y)$ (resp. $C(x)=C(y)$). We may assume
that $O=(\Fix(f_{1}),\Fix(f_{s+1}))$, where $\Fix(f_{k})$
denotes the unique fixed point of $f_{k}$ in $\R$. 

Let $x,y\in J$ with $x<y$. Let $\omega\in I^{*}\cup\{\emptyset\}$
and $i,j\in I$ with $i<j$ such that $f_{\omega}(x)\in\overline{O}$,
$f_{\omega}(y)\in\overline{O}$ and $f_{i}(f_{\omega}(x))\in\overline{O}$
and $f_{j}(f_{\omega}(y))\in\overline{O}$. Note that such $(\omega,i,j)$
always exists because $x,y\in J$ and $x<y$. 

We first consider the case when $j=i+1$. Let 
\[
\ell:=\sup\left\{ n\ge0\mid f_{s+1}^{n}(f_{i}(f_{\omega}(x)))\in\overline{O}\right\} \quad\text{and}\quad\ell':=\sup\left\{ n\ge0\mid f_{1}^{n}(f_{j}(f_{\omega}(y)))\in\overline{O}\right\} .
\]

We define 
\[
\xi:=\max(f_{\omega i})^{-1}(\overline{O})\quad\text{and}\quad\xi':=\min(f_{\omega j})^{-1}(\overline{O}).
\]
We will verify that there exists a uniform constant $D'\ge1$ such
that
\begin{equation}
\left|T_{\vec{p}}(x)-T_{\vec{p}}(\xi)\right|\le C(\varphi,\psi_{\vec{p}})D'd(x,\xi)^{\alpha_{-}(\vec{p})}\label{eq:osc estimate}
\end{equation}
and 
\begin{equation}
\left|T_{\vec{p}}(\xi')-T_{\vec{p}}(y)\right|\le C(\varphi,\psi_{\vec{p}})D'd(\xi',y)^{\alpha_{-}(\vec{p})}.\label{eq:osc estimate-1}
\end{equation}
By the triangle inequality we have that 
\[
\left|T_{\vec{p}}(x)-T_{\vec{p}}(y)\right|\le\left|T_{\vec{p}}(x)-T_{\vec{p}}(\xi)\right|+\left|T_{\vec{p}}(\xi)-T_{\vec{p}}(\xi')\right|+\left|T_{\vec{p}}(\xi')-T_{\vec{p}}(y)\right|,
\]
which proves the first assertion in the case when $j=i+1$ because
$T_{\vec{p}}$ is constant on $[\xi,\xi']$. 

We only verify (\ref{eq:osc estimate}), the proof of (\ref{eq:osc estimate-1})
is completely analogous. To prove (\ref{eq:osc estimate}) we may assume
that $\ell<\infty$ because if $\ell=\infty$ then $x=\xi$ and (\ref{eq:osc estimate})
holds trivially. By the definition of $\ell$ there exists $1\le k<s+1$
such that 
\[
x\in f_{\omega}^{-1}f_{i}^{-1}f_{s+1}^{-\ell}f_{k}^{-1}(\overline{O})\le f_{\omega}^{-1}f_{i}^{-1}f_{s+1}^{-\ell}f_{s+1}^{-1}(\overline{O})\subset f_{\omega}^{-1}f_{i}^{-1}(\overline{O}).
\]
We conclude that 
\begin{equation}
d(x,\xi)\ge\diam\left(f_{\omega}^{-1}f_{i}^{-1}f_{s+1}^{-\ell}f_{s+1}^{-1}(\overline{O})\right).\label{eq:gapinterval}
\end{equation}
Define $\tau:=\omega i(s+1)^{\ell}$. Clearly, we have $[x,\xi]\subset f_{\tau}^{-1}(\overline{O})$.
Since the open set condition holds, 
\begin{equation}
\left|T_{\vec{p}}(x)-T_{\vec{p}}(\xi)\right|\le p_{\tau}.\label{eq:estimate for Fp}
\end{equation}
By (\ref{eq:uniform psi phi estimate}) there exists a constant $C(\varphi,\psi_{\vec{p}})$
such that 
\begin{equation}\label{eq:psiphibound1}
p_{\tau}=\e^{S_{\left|\tau\right|}\psi_{\vec{p}}(\overline{\tau})}\le C(\varphi,\psi_{\vec{p}})\e^{\alpha_{-}(\vec{p})S_{|\tau|}\varphi(\overline{\tau})}.
\end{equation}
By the bounded distortion property (Lemma \ref{lem:bounded distortion})
and (\ref{eq:gapinterval}), there exists a uniform constant $D'\ge1$
such that $\e^{\alpha_{-}(\vec{p})S_{|\tau|}\varphi(\overline{\tau})}\le D'd(x,\xi)^{\alpha_{-}(\vec{p})}.$
We have thus shown (\ref{eq:osc estimate}).  It remains to consider
the case when $j>i+1$. Using that $[x,y]\subset f_{\omega}^{-1}(\overline{O})$
we obtain as above 
\[
\left|T_{\vec{p}}(x)-T_{\vec{p}}(y)\right|\le p_{\omega}\le C(\varphi,\psi_{\vec{p}})\e^{\alpha_{-}(\vec{p})S_{|\omega|}\varphi(\overline{\omega})}.
\]
Here, we set $p_{\emptyset}=1$ and $S_{|\emptyset|}\varphi
(\overline{\emptyset}):=0$. By our assumption that $j>i+1$ we have 
$d(x,y)\ge\diam(f_{\omega}^{-1}f_{i+1}^{-1}(O))$,
which implies that there exists a constant $D''\ge1$ such that 
\[
\e^{\alpha_{-}(\vec{p})S_{|\tau|}\varphi(\overline{\omega})}\le 
D''d(x,y)^{\alpha_{-}(\vec{p})}.
\]
The proof of the first assertion is complete.  To prove the second
assertion, it is sufficient to verify it for $C_{\vec{n}}$ for every
$\vec{n}\in\N_0^{s}$. To this end, we replace the estimate in (\ref{eq:estimate 
for Fp})
by the following estimate.  By Lemma \ref{lem:functional equation via cocycle}
and Lemma \ref{lem:growth of cocycle} there exists a constant $K'\ge1$
(which only depends on $\vec{p}$ and $|\vec{n}|$) such that 
\begin{equation} \label{eq:Cnbound}
\left|C_{\vec{n}}(x)-C_{\vec{n}}(\xi)\right|\le K'|\tau|^{|\vec{n}|}p_{\tau}.
\end{equation}
Then proceeding as above, for every $\alpha<\alpha_{-}(\vec{p})$
there exists $D'''\ge1$ such that 
\begin{equation}
\left|C_{\vec{n}}(x)-C_{\vec{n}}(\xi)\right|\le 
K'|\tau|^{|\vec{n}|}C(\varphi,\psi_{\vec{p}})
\e^{\alpha_{-}(\vec{p})S_{|\tau|}\varphi(\overline{\tau})}\le 
K'|\tau|^{|\vec{n}|}C(\varphi,\psi_{\vec{p}})D'''d(x,\xi)^{\alpha}\e^{(\alpha_{-}(\vec{p})-\alpha)S_{|\tau|}\varphi(\overline{\tau})}.\label{eq:hoelder
 estimate for higher order}
\end{equation}
Since $\alpha-\alpha_{-}(\vec{p})<0$, we have 
$K'|\tau|^{|\vec{n}|}C(\varphi,\psi_{\vec{p}})D'''\e^{(\alpha_{-}(\vec{p})-\alpha)S_{|\tau|}\varphi(\overline{\tau})}\rightarrow0$
uniformly as $|\tau|\rightarrow\infty$. Hence, $C_{\vec{n}}$ is
$\alpha$-H\"older continuous. 
\end{proof}
\begin{rem}
\label{rem:convex lipschitz} By the bounded distortion property of $(f_i)$ 
(Lemma \ref{lem:bounded distortion}) it is not difficult to see that 
$|\tau|\asymp-\log d(x,\xi)$. By combining this estimate with  
\eqref{eq:psiphibound1} and 
\eqref{eq:Cnbound}, 
we can derive that $|C_{\vec{n}}(x)-C_{\vec{n}}(y)|\ll(-\log 
d(x,y))^{|\vec{n}|}d(x,y)^{\alpha_{-}(\vec{p})}$.
If $\alpha_{-}=1$ this implies that $C_{\vec{n}}$ is convex Lipschitz
(\cite{MR860394}). In fact, this property was derived in \cite{MR860394}
for the classical Takagi function, and for the higher order derivatives
of the Lebesgue singular function in \cite{MR2212280}. The property
of convex Lipschitz can be used to prove that the graph of $C_{\vec{n}}$
has Hausdorff dimension one if $\alpha_{-}=1$. 
\end{rem}

Recall that we have $\alpha_{-}\le\delta:=\dim_{H}(J)$ with equality
if and only if $\alpha_{-}=\alpha_{+}$ (\cite[Chapter 21, see in particular, Figure 17b]{pesindimensiontheoryMR1489237}).
\begin{cor}
\label{cor:optimal hoelder continuity}Suppose that $(f_{i})_{i\in I}$
satisfies the open set condition.  Then for every non-trivial  $C\in\T$
we have 
\[
\alpha_{-}=\sup\left\{ \alpha\ge0\mid C\in\mathcal{C}^{\alpha}(\overline{\R})\right\} \le\delta.
\]
The equality $\alpha_{-}=\alpha_{+}$ occurs if and only if $T_{\vec{p}}\in\mathcal{C}^{\delta}(\overline{\R})$. 
\end{cor}

\begin{proof}
The first assertion follows from Theorem  \ref{thm:multifrac}
and Theorem \ref{thm:alpha-minus hoelder}. Now suppose that $\alpha_{-}=\alpha_{+}$.
Hence, $\alpha_{-}=\delta$ and $T_{\vec{p}}\in\mathcal{C}^{\delta}(\overline{\R})$
by Theorem \ref{thm:alpha-minus hoelder}. Conversely, suppose that
$T_{\vec{p}}\in\mathcal{C}^{\delta}(\overline{\R})$. Then, by the
first assertion of Corollary \ref{cor:optimal hoelder continuity},
we have $\alpha_{-}\ge\delta$. Hence, $\alpha_{-}=\alpha_{+}$. 
\end{proof}
\begin{cor}
\label{cor:transition operator chaotic on some Calpha}Suppose that
$(f_{i})_{i\in I}$ satisfies the open set condition. If $\alpha_{-}<1$
then, for each $1>\alpha>\alpha_{-}$ there exists $\phi\in\mathcal{C}^{\alpha}\left(\overline{\R}\right)$
such that $\Vert M_{\vec{p}}^{n}\phi\Vert_{\alpha}\rightarrow\infty$,
as $n\rightarrow\infty$. 
\end{cor}

\begin{proof}
Let $\phi\in\mathcal{C}^{\alpha}\left(\overline{\R}\right)$ such
that $\phi(\infty)=1$ and $\phi(-\infty)=0$. Then $\Vert M_{\vec{p}}^{n}\phi-T_{\vec{p}}\Vert_{\infty}\rightarrow0$
as $n\rightarrow\infty$. Now suppose for a contradiction that 
$\liminf_{n\rightarrow \infty}\Vert M_{\vec{p}}^{n}\phi\Vert_{\alpha}<\infty$.
Then there exists a sequence $(n_j)$ in $\N$  tending to infinity and a 
constant $M<\infty$ such that $\Vert M_{\vec{p}}^{n_j}\phi\Vert_\alpha <M$ for 
each $j$.  Hence, $|M_{\vec{p}}^{n_j}\phi(x)-M_{\vec{p}}^{n_j}\phi(y)|\le 
Md(x,y)^\alpha$ for all $j$ and $x,y\in \R$. Letting $j\rightarrow \infty$ 
gives $|T_{\vec{p}}(x)-T_{\vec{p}}(y) |\le M d(x,y)^\alpha$ for all $x,y\in \R$. 
We 
have thus shown that 
$T_{\vec{p}}\in\mathcal{C}^{\alpha}\left(\overline{\R}\right)$
giving the desired contradiction to Theorem \ref{intro thm:global hoelder}. 
\end{proof}
In the following, we will denote $C_{(n)}$ by $C_{n}$ if $s=1$
and $n\in\N_{0}$. 
\begin{prop}
\label{prop:derivatives are not alpha- hoelder}Suppose that $(f_{i})_{i\in I}$
satisfies the open set condition. 
\begin{enumerate}
\item If $\alpha_{-}=\alpha_{+}$ then $\T\cap\mathcal{C}^{\alpha_{-}}(\overline{\R})=\R T_{\vec{p}}$. 
\item If $s=1$ and each $f_{i}$ has constant derivative, then $C_{n}\notin$
$\mathcal{C}^{\alpha_{-}}(\overline{\R})$ for every $n\ge1$.
\end{enumerate}
\end{prop}

\begin{proof}
We first prove (1). For $\vec{n},\vec{m}\in\N_{0}^{s}$ we use $\preceq$ to 
denote the lexicographical order, that is, we write   
$\vec{n}\preceq\vec{m}$ if $n_k<m_k$ where  $k:=\min \{1\le i \le s\mid n_i 
\neq m_i \}$ or if $\vec{n}=\vec{m}$. 
Since the lexicographical order is a total order, each non-trivial $C\in \T$ 
has a representation  as
$C=\sum_{\vec{n}\preceq \vec{n}_{\max}}\beta_{\vec{n}}C_{\vec{n}}$
with $\beta_{\vec{n}_{\max}}\neq0$. Suppose that  
$n:=\sum_{i=1}^{s}(\vec{n}_{\max})_{i}\ge1$.
Our aim is to prove that $C\notin\mathcal{C}^{\alpha_{-}}(\overline{\R})$.
Let $a:=\min J$ and $b:=\max J$. Define for $k\ge1$ 
\[
m_{i}(k):=k^{(n+1)^{i-1}},\quad1\le i\le s,
\]
and for $\omega(k):=1^{m_{1}(k)}2^{m_{2}(k)}\dots s^{m_{s}(k)}\in I^{\sum_{i=1}^{s}m_{i}(k)}$
we let 
\[
x_{k}:=\left(f_{\omega(k)}\right)^{-1}(a),\quad y_{k}:=\left(f_{\omega(k)}\right)^{-1}(b).
\]
Recall that, since $\alpha_{-}=\alpha_{+}$, we have $\alpha_{-}=\alpha_{+}=\delta=\dim_{H}J$,
and the potentials $\delta\varphi$ and $\psi$ are cohomologous.
It follows that 
\[
\frac{\left|C(x_{k})-C(y_{k})\right|}{d(x_{k},y_{k})^{\alpha_{-}}}\asymp\frac{\left|C(x_{k})-C(y_{k})\right|}{p_{\omega(k)}},\quad\text{as }k\rightarrow\infty.
\]
Since $C_{\vec{n}}(a)-C_{\vec{n}}(b)=\vec{0}$, for $\vec{n}\neq\vec{0}$
and $C_{\vec{0}}(a)-C_{\vec{0}}(b)=1$, we conclude  by Lemma
\ref{lem:functional equation via cocycle} that
\begin{align*}
\frac{C(x_{k})-C(y_{k})}{p_{\omega(k)}} & 
=\sum_{\vec{n}\preceq\vec{n}_{\max}}\beta_{\vec{n}}
\left(A\left(\overline{\omega(k)},\sum_{i=1}^{s}m_{i}(k)\right)U(a,b)\right)_{\vec{n}}
=\sum_{\vec{n}\preceq\vec{n}_{\max}}\beta_{\vec{n}} 
\left(A\left(\overline{\omega(k)},\sum_{i=1}^{s}m_{i}(k)\right)\right)_{\vec{n},\vec{0}}.
\end{align*}
Hence, by Lemma \ref{lem:growth of cocycle} we have
as $k\rightarrow\infty$ 
\[
\frac{\left|C(x_{k})-C(y_{k})\right|}{p_{\omega(k)}}\asymp
\left|\sum_{\vec{n}\preceq\vec{n}_{\max}}\beta_{\vec{n}}\prod_{i=1}^{s}m_{i}(k)^{n_{i}}
\right|\asymp\left|\sum_{\vec{n}\preceq\vec{n}_{\max}}\beta_{\vec{n}}k^{\sum_{i=1}^{s}n_{i}(n+1)^{i-1}}\right|\asymp\left|\beta_{\vec{n}_{\max}}\right|k^{\sum_{i=1}^{s}(\vec{n}_{\max})_{i}(n+1)^{i-1}}.
\]
It follows that $\left|C(x_{k})-C(y_{k})\right|/d(x_{k},y_{k})^{\alpha_{-}}$
is not bounded, as $k\rightarrow\infty$, which implies that $C\notin\mathcal{C}^{\alpha_{-}}(\overline{\R})$.
The proof of (1) is complete.

To prove (2) we can proceed along the same lines. First note that,
by our assumptions on $f_{1}$ and $f_{2}$, $\varphi$
depends only on the first coordinate. We show that  there exists $i\in I$ such 
that 
$\psi(\overline{i})/\varphi(\overline{i})=\alpha_{-}$. First observe that there 
exists  $i\in I$ such that $\psi(\overline{i})/\varphi(\overline{i})=\min 
(\psi / \varphi)$. Thus, for all $n\ge 1$,  $S_n\psi(\overline{i}) / S_n 
\varphi(\overline{i}) =\min 
(\psi / \varphi) $. From the  mediant inequality we derive inductively that  
$S_n\psi / S_n \varphi \ge \min 
(\psi / \varphi) $ which completes the proof that 
$\alpha_-=\psi(\overline{i})/\varphi(\overline{i})$. 
Hence, for the sequences given by $x_{k}:=\left(f_{i}\right)^{-k}(a)$
and $y_{k}:=\left(f_{i}\right)^{-k}(b)$ we have as $k\rightarrow\infty$,
\[
\frac{\left|C_{n}(x_{k})-C_{n}(y_{k})\right|}
{d(x_{k},y_{k})^{\alpha_{-}}}\asymp
\frac{\left|C_{n}(x_{k})-C_{n}(y_{k})\right|}
{p_{i}^{k}}\asymp\left|\left(A(\overline{i},k)U(a,b)\right)_{n}\right|\asymp
 k^{n},
\]
which tends to infinity, as $k\rightarrow\infty$. Again this implies
that $C_{n}\notin\mathcal{C}^{\alpha_{-}}(\overline{\R})$. The proof
is complete.
\end{proof}

\section{Non-differentiability\label{sec:Non-differentiability}}

In this section we investigate the (non-)differentiability of the
elements of $\T$. Note that $\alpha_{-}\le\dim_{H}(J)\le1$. We denote
by Leb the Lebesgue measure on $\left[0,1\right]$. 
\begin{prop}
\label{prop:non-differentiability}Suppose that $(f_{i})_{i\in I}$
satisfies the open set condition.  
\begin{enumerate}
\item If $\alpha_{-}<1$ then there exists a dense subset $E\subset J$
with $\dim_{H}(E)>0$ such that, for every non-trivial $C\in\T$
and $x\in E$, $C$ is not differentiable at $x$. If moreover $\alpha_{+}<1$
then $C\in\T\setminus\{0\}$ is nowhere differentiable on $J$. 
\item If $\alpha_{-}<1$ and $\dim_{H}(J)=1$ then we have, for every $C\in\T$,
$C'(x)=0$ for Leb-a.e. $x\in J$.
\item If $\alpha_{-}=1$ then $\alpha_{+}=1$ and $\dim_{H}(J)=1$. 
If moreover $s=1$ and $f_1'$ and $f_2'$ are constant functions, then $C_{m}$ is 
nowhere differentiable
on $J$, for every $m\ge1$.  
\end{enumerate}
\end{prop}

\begin{proof}
First assume that $\alpha_{-}<1$. If $\alpha_{-}<\alpha_{+}$ then
there exists $\alpha\in(\alpha_{-},\alpha_{+})$ with $\alpha<1$.
Then by Proposition \ref{prop:upperbound pointwise hoelder} we have
$\Hol(C,x)\le\alpha<1$ for all $x\in E:=\pi(\mathcal{F}(\alpha))$,
and $E$ has the desired properties. If moreover $\alpha_{+}<1$ then
we have $\Hol(C,x)\le\alpha_{+}<1$ for every $x\in J$ by Proposition
\ref{prop:upperbound pointwise hoelder}, which completes the proof
of (1). 

To prove (2), let $\alpha_{-}<1$ and observe that $\dim_{H}(J)=1$
implies $\alpha_{-}<\alpha_{+}$. Hence, $t$ is strictly convex.
Since $t(0)=1$ and $t(1)=0$, it follows that $-t'(0)>1$. Thus,
we have $\Hol(C,x)=\int\psi\,\,d\mu_{0}/\int\varphi\,\,d\mu_{0}=-t'(0)>1$
$\mu_{0}$-almost everywhere. Since $\mu_{0}$ is equivalent to Leb,
the assertion in (2) follows. 

Finally, we turn to the proof of (3). Let  $s=1$.  Recall that $\alpha_{-}=1$ 
implies $\alpha_{-}=\alpha_{+}=\dim_H(J)=1$ and
that $\varphi$ is cohomologous to $\psi$. Hence, there exists a continuous 
function $h:\Sigma \rightarrow \R$ such that $\varphi = \psi + h - h\circ 
\sigma $. Since $f_1'$ and $f_2'$ are constant 
functions, $\varphi=\varphi(\omega)$ depends only on the first symbol 
$\omega_1$ of $\omega=(\omega_1,\omega_2,\dots) \in \Sigma$. We conclude that 
for 
every $\omega \in \Sigma$, 
\[ \varphi(\omega)=\varphi(\overline{\omega_1})=\psi(\overline{\omega_1}) + 
h(\overline{\omega_1}) - h\circ \sigma (\overline{\omega_1})=\psi(\omega).
\]
Now,   suppose for a contradiction
that there exists $m\ge1$ and  $x\in J$ such that $C_m$ is differentiable
at $x$ and let $\omega\in\Sigma$ such that $x=\pi(\omega)$. Let
$a:=\min J$ and $b:=\max J$. Define $x_{n}:=(f_{\omega|n})^{-1}(a)$
and $y_{n}:=(f_{\omega|n})^{-1}(b)$, $n\ge1$. We will verify that
the sequence $(\gamma_{n})_{n\ge1}$ given by 
\[
\gamma_{n}:=\frac{C_m(x_{n})-C_m(y_{n})}{x_n -y_n },\quad n\ge1,
\]
is not convergent. Since $x_n \le x \le y_n$ and  $x_{n},y_{n}\rightarrow x$ we 
obtain the
desired contradiction. To prove that $(\gamma_{n})$ is not convergent,
note that for all $n\ge 1$,
\[\frac{p_{\omega_{|n}}}{|x_{n}-y_{n}|}=\frac{
\e^{S_{n}\psi(\omega)-S_{n}\varphi(\omega)}}{b-a}=\frac{1}{b-a}.
\]
Combining with  $C_m(a)=C_m(b)=0$, 
$T(a)-T(b)=1$ and Lemma \ref{lem:functional equation via cocycle}
we obtain
\begin{align*}
\gamma_{n}\cdot(b-a) 
=- \frac{\left(U(x_{n},y_{n})\right)_m}{p_{\omega_{|n}}}=-
\left(A(\omega,n)U(a,b)\right)_m=-A(\omega,n)_{m,0}(T(a)-T(b))=-A(\omega,n)_{m,0}.
\end{align*}
Let $\delta_{i,j}$ denote the Dirac delta function, for $i,j\in I$.
By the definition of the matrix cocycle $A$ we have 
\begin{align}
A(\omega,n)_{1,0} 
=\sum_{i=1}^{n}A(\sigma^{i-1}(\omega),1)_{1,0}=
 \sum_{i=1}^{n}\left(\frac{\delta_{1,\omega_{i}}}{p_{1}}-
 \frac{\delta_{2,\omega_{i}}}{p_{2}}\right).\label{eq:nondiff}
\end{align}
This shows that $(A(\omega,n)_{1,0})$ does not converge as $n\rightarrow 
\infty$ because $\frac{\delta_{1,\omega_{i}}}{p_{1}}-
\frac{\delta_{2,\omega_{i}}}{p_{2}}\in \{1/p_1, -1/p_2 \}$. We now proceed 
inductively to verify that 
$(A(\omega,n)_{m,0})$ does 
not converge as $n\rightarrow \infty$.  It
is easy to see that 
\begin{align*}
A(\omega,n+1)_{m,0} & =A(\omega,n)_{m,0}A(\sigma^{n}(\omega),1)_{0,0}+A(\omega,n)_{m,1}A(\sigma^{n}(\omega),1)_{1,0}\\
 & =A(\omega,n)_{m,0}+A(\omega,n)_{m,1}A(\sigma^{n}(\omega),1)_{1,0}.
\end{align*}
Hence, 
\[
A(\omega,n+1)_{m,0}-A(\omega,n)_{m,0}=A(\omega,n)_{m,1}\cdot A(\sigma^{n}(\omega),1)_{1,0}=A(\omega,n)_{m,1}\cdot\left(\frac{\delta_{1,\omega_{n+1}}}{p_{1}}-\frac{\delta_{2,\omega_{n+1}}}{p_{2}}\right).
\]
Further, we have 
\begin{align*}
A(\omega,n)_{m,1} & =\sum_{1\le i_{1}<\dots<i_{m-1}\le n}\prod_{k=1}^{m-1}A(\sigma^{i_{k}-1}(\omega),1)_{m-k+1,m-k}\\
 & =\sum_{1\le i_{1}<\dots<i_{m-1}\le n}\prod_{k=1}^{m-1}\left(\frac{\delta_{1,\omega_{i_{k}}}}{p_{1}}-\frac{\delta_{2,\omega_{i_{k}}}}{p_{2}}\right)(m-k+1)\\
 & =m\cdot\sum_{1\le i_{1}<\dots<i_{m-1}\le n}\prod_{k=1}^{m-1}\left(\frac{\delta_{1,\omega_{i_{k}}}}{p_{1}}-\frac{\delta_{2,\omega_{i_{k}}}}{p_{2}}\right)(m-k)\\
 & =m\cdot A(\omega,n)_{m-1,0}.
\end{align*}
Therefore, 
\begin{equation}\label{nondiff1}
A(\omega,n+1)_{m,0}-A(\omega,n)_{m,0}=m\cdot A(\omega,n)_{m-1,0}\cdot
\left(\frac{\delta_{1,\omega_{n+1}}}{p_{1}}-\frac{\delta_{2,\omega_{n+1}}}{p_{2}}\right).
\end{equation}
By induction hypothesis, we may assume that $(A(\omega,n)_{m-1,0})$ is not 
convergent as $n \rightarrow \infty$. Since 
$\frac{\delta_{1,\omega_{i}}}{p_{1}}-
\frac{\delta_{2,\omega_{i}}}{p_{2}}\in \{1/p_1, -1/p_2 \}$
 we conclude by (\ref{nondiff1}) that $(A(\omega,n)_{m,0})$ is not 
 convergent 
as $n\rightarrow \infty$. The proof of (3) is
complete.

\end{proof}
\begin{rem}
If $\alpha_{-}=\alpha_{+}=1$, then the nowhere-differentiability
of elements of $\T$ stated in Theorem \ref{prop:non-differentiability}
(3) will be compared with the fact that $T_{\vec{p}}$ is a  
$\mathcal{C}^{1+\epsilon}$-diffeomorphism
by Corollary \ref{cor:conjugacy}.
\end{rem}

\section{Conjugacies between Interval 
maps\label{sec:Conjugaries-between-Interval}}

In this section we show how our results are related to interval conjugacies. 
This section is motivated by the results in \cite{MR2576266}
for conjugacies between expanding $C^{1+\epsilon}$ maps on the
unit interval with finitely many full branches. Note that in \cite{MR2576266}
it is always assumed that the Julia set $J$ is equal to the unit
interval. 

For $\vec{p}\in(0,1)^{s}$ we define the expanding linear maps $g_{1},\dots,g_{s+1}:\overline{\R}\rightarrow\overline{\R}$
which are for $i\in I$ given by 
\[
g_{i}(x):=\frac{1}{p_{i}}\left(x-\sum_{1\le j<i}p_{j}\right).
\]
Clearly, $(g_{1},\dots,g_{s+1})$ satisfies our standing assumptions.
Moreover, $(g_{1},\dots,g_{s+1})$ satisfies the open set condition
with the  open set $O:=(0,1)$ because 
$g_{i}^{-1}(O)=(\sum_{j<i}p_{j},\sum_{j\le i}p_{j})$.
The Julia set of $(g_{1},\dots,g_{s+1})$ is equal to $[0,1]$. We
denote by $\pi_{\vec{p}}:\Sigma\rightarrow[0,1]$ the coding map of
the Julia set of $(g_{1},\dots,g_{s+1})$. Similarly, we define 
\[
g_{\vec{p}}:[0,1]\rightarrow[0,1],\quad g_{\vec{p}}(x):=g_{i}(x),\quad\text{where }i=\min\left\{ j\in I\mid g_{j}(x)\in[0,1]\right\} .
\]
Note that $g_{\vec{p}}$ is the piecewise linear map on $[0,1]$ with
$(s+1)$ full branches and slopes given by $(1/p_{i})_{i\in I}.$ 

Suppose that $(f_{1},\dots,f_{s+1})$ satisfies the open set condition
and suppose that $f_{i}^{-1}(O)\le f_{i+1}^{-1}(O)$ for all $i\in I$.
Denote the Julia set of $(f_{1},\dots,f_{s+1})$ by $J$ and its coding
map by $\pi:\Sigma\rightarrow J$. We also define 
\[
f:J\rightarrow J,\quad f(x):=f_{i}(x),\quad\text{where }i=\min\left\{ j\in I\mid f_{j}(x)\in J\right\} .
\]

Further, we define 
\[
\Phi_{\vec{p}}:J\rightarrow[0,1],\quad\Phi_{\vec{p}}(x):=\pi_{\vec{p}}(\omega)\text{ for some / any }\omega\in\pi^{-1}(x).
\]
Note that $\Phi_{\vec{p}}$ is a well-defined, Borel measurable function
satisfying $\Phi_{\vec{p}}(\max J)=1$ and $\Phi_{\vec{p}}(\min J)=0$.
Let $\mathcal{E}:=\left\{ \omega\in\Sigma\mid\omega\text{ is eventually constant}\right\} $.
For $x\in\pi(\Sigma\setminus\mathcal{E})$ we denote by $\pi^{-1}(x)$
the unique $\omega\in\Sigma\setminus\mathcal{E}$ such that $\pi(\omega)=x$.
For $\omega\in\Sigma\setminus\mathcal{E}$ and $x=\pi(\omega)$ we
have $f(x)=f_{\omega_{1}}(x)$ and thus, 
\[
\pi^{-1}(f(x))=\sigma(\omega).
\]
Further, $\vec{g}_{\vec{p}}(\pi_{\vec{p}}(\omega))=g_{\omega_{1}}(\pi_{\vec{p}}(\omega))$
implies 
\[
\pi_{\vec{p}}\circ\sigma(\omega)=\vec{g}_{\vec{p}}\circ\pi_{\vec{p}}(\omega).
\]
Hence, for $x\in\pi(\Sigma\setminus\mathcal{E})$, 
\begin{equation}
\Phi_{\vec{p}}(f(x))=\pi_{\vec{p}}\circ\pi^{-1}(f(x))=\pi_{\vec{p}}\circ\sigma(\omega)=\vec{g}_{\vec{p}}\circ\pi_{\vec{p}}(\omega)=\vec{g}_{\vec{p}}(\Phi_{\vec{p}}(x)).\label{eq:conjugacy condition}
\end{equation}

If $J$ is an interval, the following lemma appears implicitly in
\cite[Proof of Proposition 1.4]{MR2576266}.
\begin{lem}
\label{lem:T is conjugacy}For every $\vec{p}\in(0,1)^{s}$ we have
that $\Phi_{\vec{p}}=T_{\vec{p}|J}$. 
\end{lem}

\begin{proof}
We proceed in two steps. First, we will show that $\Phi_{\vec{p}}(x)=T_{\vec{p}}(x)$
for $x\in\pi(\Sigma\setminus\mathcal{E})$. Let $\tilde{\Phi}_{\vec{p}}:\overline{\R}\rightarrow[0,1]$
denote a bounded Borel measurable extension of $\Phi_{\vec{p}}$ such
that  $\tilde{\Phi}_{\vec{p}}(y)=0$ for $y\in[-\infty,\min J]$,
and $\tilde{\Phi}_{\vec{p}}(y)=1$ for $y\in[\max J,+\infty]$. Let
$\omega\in\Sigma\setminus\mathcal{E}$.  For $i<\omega_{1}$ we have
$f_{i}(\pi(\omega))\ge\max J$, and for $i>\omega_{1}$ we have $f_{i}(\pi(\omega))\le\min J$.
Since $\Phi_{\vec{p}}(\max J)=1$ and $\Phi_{\vec{p}}(\min J)=0$,
the equality in \eqref{eq:conjugacy condition} yields for $x=\pi(\omega)$
\begin{align*}
M_{\vec{p}}(\tilde{\Phi}_{\vec{p}})(x) & =\sum_{i\in I}p_{i}\tilde{\Phi}_{\vec{p}}(f_{i}(x))=\sum_{i<\omega_{1}}p_{i}+p_{\omega_{1}}\cdot\Phi_{\vec{p}}(f(x))+\sum_{i>\omega_{1}}0=\sum_{i<\omega_{1}}p_{i}+p_{\omega_{1}}\cdot\vec{g}_{\vec{p}}(\Phi_{\vec{p}}(x)).
\end{align*}
Since $x=\pi(\omega)$ and $\omega\in\Sigma\setminus\mathcal{E}$,
we have $\vec{g}_{\vec{p}}(\Phi_{\vec{p}}(x))=g_{\omega_{1}}(\Phi_{\vec{p}}(x))$.
Hence, we have for every $x\in\pi(\Sigma\setminus\mathcal{E})$, 
\[
M_{\vec{p}}(\tilde{\Phi}_{\vec{p}})(x)=\sum_{i<\omega_{1}}p_{i}+p_{\omega_{1}}\cdot g_{\omega_{1}}(\Phi_{\vec{p}}(x))=\sum_{i<\omega_{1}}p_{i}+p_{\omega_{1}}\cdot\frac{1}{p_{\omega_{1}}}\left(\Phi_{\vec{p}}(x)-\sum_{i<\omega_{1}}p_{i}\right)=\tilde{\Phi}_{\vec{p}}(x).
\]
Further, for every $x\in[-\infty,\min J]\cup[\max J,+\infty]$ we
have $M_{\vec{p}}(\tilde{\Phi}_{\vec{p}})(x)=\tilde{\Phi}_{\vec{p}}(x)$.
Let $E:=\pi(\Sigma\setminus\mathcal{E})\cup[-\infty,\min J]\cup[\max J,+\infty]$.
Since $f_{i}(E)\subset E$ for every $i\in I$, we can show inductively
that for every $x\in E$ and $n\in\N$, 
\[
M_{\vec{p}}^{n}(\tilde{\Phi}_{\vec{p}})(x)=M_{\vec{p}}(M_{\vec{p}}^{n-1}\tilde{\Phi}_{\vec{p}})(x)=\sum_{i\in I}p_{i}(M_{\vec{p}}^{n-1}\tilde{\Phi}_{\vec{p}})(f_{i}(x))=\sum_{i\in I}p_{i}\tilde{\Phi}_{\vec{p}}(f_{i}(x))=\tilde{\Phi}_{\vec{p}}(x).
\]

By \eqref{eq:Mp convergence measurable} (Remark: \eqref{eq:Mp convergence 
measurable} is valid for any bounded measurable function $h$ on $\overline{R}$ 
such that $h=1$ around $+\infty$ and $h=0$ around $-\infty$) we conclude that 
for $x\in\pi(\Sigma\setminus\mathcal{E})$,
\[
\tilde{\Phi}_{\vec{p}}(x)=\lim_{n\rightarrow\infty}M_{\vec{p}}^{n}\tilde{\Phi}_{\vec{p}}(x)=T_{\vec{p}}(x)\tilde{\Phi}_{\vec{p}}(\infty)+(1-T_{\vec{p}}(x))\tilde{\Phi}_{\vec{p}}(-\infty)=T_{\vec{p}}(x).
\]
This completes the proof of $\Phi_{\vec{p}}(x)=T_{\vec{p}}(x)$ for
$x\in\pi(\Sigma\setminus\mathcal{E})$. Now, let $\omega\in\mathcal{E}$
and $x=\pi(\omega)$. Let $(\omega^{(n)})\subset\Sigma\setminus\mathcal{E}$
such that $(\omega_{1}^{(n)},\dots,\omega_{n}^{(n)})=(\omega_{1},\dots,\omega_{n})$
and $x_{n}:=\pi(\omega^{(n)})$, $n\ge1$. By the continuity of $\pi$
with respect to the word metric, we have $x_{n}\rightarrow x$ as
$n\rightarrow\infty$. By the definition of $\Phi_{\vec{p}}$ we have
$\Phi_{\vec{p}}(x_{n})=\pi_{\vec{p}}(\omega^{(n)})$ and $\Phi_{\vec{p}}(x)=\pi_{\vec{p}}(\omega)$.
So, by the continuity of $\pi_{\vec{p}}$ with respect to the word
metric, we have $\Phi_{\vec{p}}(x_{n})\rightarrow\Phi_{\vec{p}}(x)$
as $n\rightarrow\infty$. Since $\Phi_{\vec{p}}(x_{n})=T_{\vec{p}}(x_{n})$
by the first part of the proof, and since $T_{\vec{p}}$ is continuous
by Theorem \ref{thm:spectralgap}, we can conclude that $\Phi_{\vec{p}}(x)=T_{\vec{p}}(x)$.
The proof is complete. 
\end{proof}
\begin{thm}[{\cite[Theorem 1.2]{MR2576266}}]
\label{thm:conjugacy theorem}Suppose that $(f_{i})_{i\in I}$ are
$\mathcal{C}^{1+\epsilon}$-diffeomorphisms satisfying the open set
condition.  Suppose that $J$ is an interval. Then for every $\vec{p}\in(0,1)^{s}$
the following rigidity dichotomy holds. 
\end{thm}

\begin{enumerate}
\item If $\alpha_{-}=\alpha_{+}(=1)$ then $\Phi_{\vec{p}}$ is a 
$\mathcal{C}^{1+\epsilon}$-diffeomorphism.
\item If $\alpha_{-}<\alpha_{+}$ then $\Phi_{\vec{p}}'\equiv0$ Leb-a.e.,
$\Phi_{\vec{p}}\in\mathcal{C}^{\alpha_{-}}(\overline{\R})$ and the set
of non-differentiable	 points of $\Phi_{\vec{p}}$ has positive Hausdorff
dimension. 
\end{enumerate}

Combining the previous theorem with the fact that $T_{\vec{p}}\in\mathcal{C}^{\alpha_{-}(\vec{p})}(J)$
(see Theorem \ref{thm:alpha-minus hoelder}) we obtain the following
corollary from Lemma \ref{lem:T is conjugacy}, Theorem \ref{thm:conjugacy theorem}
and Corollary \ref{cor:optimal hoelder continuity}. Recall that we
have $\alpha_{-}(\vec{p})\le\dim_{H}(J)$ with equality if and only
if $\alpha_{-}=\alpha_{+}$. 
\begin{cor}
\label{cor:conjugacy}Suppose that $(f_{i})_{i\in I}$ are  
$\mathcal{C}^{1+\epsilon}$-diffeomorphisms
satisfying the open set condition. Let $\delta:=\dim_{H}(J)$. For
$\vec{p}\in(0,1)^{s}$ we have $\alpha_{-}(\vec{p})=\alpha_{+}(\vec{p})$
if and only if $T_{\vec{p}}$ is $\mathcal{C}^{\delta}(\overline{\R})$.
If $\alpha_{-}(\vec{p})=1$ then $\alpha_{-}=\alpha_{+}=1$, $J=\overline{O}$
and $T_{\vec{p}}$ is a $\mathcal{C}^{1+\epsilon}$-diffeomorphism.
\end{cor}

\section{Appendix: Contractions near infinity\label{sec:Contractions-near-infinity}}

The property that $(f_{i})_{i\in I}$ is contracting near infinity
depends on the choice of the metric $d$. We will show that, by modifying
the $(f_{i})_{i\in I}$ near infinity, we can always assume that an
 expanding family $(f_{i})_{i\in I}$ is contracting near infinity
with respect to a metric $d$ which is strongly equivalent to the
Euclidean metric on compact subsets of $\R$. 

We consider the metric $d$ on $\overline{\R}$ induced by the bijection
\[
h:\overline{\R}\rightarrow\left[-1,1\right],\quad h(x):=\frac{x}{1+\left|x\right|}.
\]
Namely, we set 
\[
d\left(x,y\right):=\left|h(x)-h(y)\right|.
\]
Note that the metric $d$ generates the topology of the two-point
compactification of $\R$. Moreover, $d$ is strongly equivalent to
the Euclidean metric on compact subsets of $\R$. 

Now suppose that $(f_{i})_{i\in I}$ is expanding with expansion rate
$\lambda>1$. We can take $g_{i}$, $i\in I$, such that $g_{i}=f_{i}$
in a neighbourhood of $\overline{\R}\setminus(V_{+}\cup V_{-})$ and
\[
g_{i}'(x)\rightarrow\lambda,\quad\text{as }x\rightarrow\pm\infty.
\]

Then we can prove the following lemma.
\begin{lem}
\label{lem:fi lipschitz}There exist neighbourhoods $V^{\pm}$ of
$\pm\infty$ such that $\Lip(g_{i|V^{\pm}})<1$ for each $i\in I$.
\end{lem}

\begin{proof}
Let $i\in I$. For $x,y\in\R$ if $x,y$ are close to $\infty$ or
$x,y$ are close to $-\infty$, then we have 
\begin{align}
\frac{d(g_{i}(x),g_{i}(y))}{d(x,y)} & =\frac{\left|\left(1+\left|x\right|\right)\left(1+\left|y\right|\right)\right|}{\left(1+\left|g_{i}(x)\right|\right)\left|\left(1+\left|g_{i}(y)\right|\right)\right|}\left|\frac{g_{i}(x)-g_{i}(y)}{x-y}\right|.\label{eq:lipschitz calculation}
\end{align}
By our assumptions we have $\limsup_{u\rightarrow\infty}u/g_{i}(u)\le\lambda^{-1}$
and $\lim_{u\rightarrow\infty}g_{i}'(u)\rightarrow\lambda$. Hence,
it follows from (\ref{eq:lipschitz calculation}) that for $x,y$
sufficiently large, 
\[
\frac{d(g_{i}(x),g_{i}(y))}{d(x,y)}<1.
\]
It remains to consider the case when $x=\infty$. The case when $x=-\infty$
is similar and therefore omitted. For $y$ sufficiently large such
that $g_{i}(y)\ge(\lambda-\eta)y$, for some $\eta$ with $\lambda-\eta>1$,
we have  
\[
\frac{d(g_{i}(\infty),g_{i}(y))}{d(\infty,y)}=\frac{1+y}{1+g_{i}(y)}\le\frac{1+y}{1+(\lambda-\eta)y}<1.
\]
Hence, there exists a neighbourhood $V^{+}$ of $+\infty$ such that
$\Lip(g_{i|V^{+}})<1$. 
\end{proof}
\def\cprime{$'$}
\providecommand{\bysame}{\leavevmode\hbox to3em{\hrulefill}\thinspace}
\providecommand{\MR}{\relax\ifhmode\unskip\space\fi MR }
\providecommand{\MRhref}[2]{%
  \href{http://www.ams.org/mathscinet-getitem?mr=#1}{#2}
}
\providecommand{\href}[2]{#2}


\begin{thebibliography}{JKPS09}

\bibitem[AK06]{MR2212280}
P.~C. Allaart and K.~Kawamura, \emph{Extreme values of some continuous nowhere
  differentiable functions}, Math. Proc. Cambridge Philos. Soc. \textbf{140}
  (2006), no.~2, 269--295. \MR{2212280}

\bibitem[All17]{Allaart17}
P.~C. Allaart, \emph{Differentiability and {H}\"older spectra of a class of
  self-affine functions}, Adv. Math. \textbf{328}
(2018), 1--39.

\bibitem[BKK16]{Barany-etal16}
B.~Barany, G.~Kiss, and I.~Kolossvary, \emph{Pointwise regularity of
  parameterized affine zipper fractal curves}, Nonlinearity \textbf{31}
(2018), 1705--1733.

\bibitem[Fal03]{falconerfractalgeometryMR2118797}
K.~Falconer, \emph{Fractal geometry}, second ed., John Wiley \& Sons Inc.,
  Hoboken, NJ, 2003, Mathematical foundations and applications. \MR{MR2118797
  (2006b:28001)}

\bibitem[HY84]{MR839313}
M.~Hata and M.~Yamaguti, \emph{The {T}akagi function and its generalization},
  Japan J. Appl. Math. \textbf{1} (1984), no.~1, 183--199. \MR{839313}

\bibitem[ITM50]{MR0037469}
C.~T. Ionescu~Tulcea and G.~Marinescu, \emph{Th\'eorie ergodique pour des
  classes d'op\'erations non compl\`etement continues}, Ann. of Math. (2)
  \textbf{52} (1950), 140--147. \MR{0037469}

\bibitem[JKPS09]{MR2576266}
T.~Jordan, M.~Kesseb{\"o}hmer, M.~Pollicott, and B.~O. Stratmann, \emph{Sets of
  nondifferentiability for conjugacies between expanding interval maps}, Fund.
  Math. \textbf{206} (2009), 161--183. \MR{2576266}

\bibitem[JS15]{JS13b}
J.~Jaerisch and H.~Sumi, \emph{Multifractal formalism for expanding rational
  semigroups and random complex dynamical systems}, Nonlinearity \textbf{28}
  (2015), 2913--2938.

\bibitem[JS17]{JS2016}
\bysame, \emph{Pointwise {H}\"older exponents of the complex analogues of the
  {T}akagi function in random complex dynamics}, Adv. Math. \textbf{313}
  (2017), 839--874.
  
  \bibitem[JS20]{JS20}
  \bysame, \emph{Multifractal analysis of generalised {T}akagi functions on 
  	the real line}, to appear in RIMS Kokyuroku. 
  

\bibitem[Kat76]{katoperturbationMR0407617}
T.~Kato, \emph{Perturbation theory for linear operators}, second ed.,
  Springer-Verlag, Berlin, 1976, Grundlehren der Mathematischen Wissenschaften,
  Band 132. \MR{MR0407617 (53 \#11389)}

\bibitem[KS08]{MR2444218}
M.~Kesseb{{\"o}}hmer and B.~O. Stratmann, \emph{Fractal analysis for sets of
  non-differentiability of {M}inkowski's question mark function}, J. Number
  Theory \textbf{128} (2008), no.~9, 2663--2686. \MR{MR2444218}

\bibitem[KS09]{MR2525939}
\bysame, \emph{H{\"o}lder -differentiability of {G}ibbs distribution
  functions}, Math. Proc. Cambridge Philos. Soc. \textbf{147} (2009), no.~2,
  489--503. \MR{MR2525939}

\bibitem[Lju83]{MR741393}
M.~J. Ljubich, \emph{Entropy properties of rational endomorphisms of the
  {R}iemann sphere}, Ergodic Theory Dynam. Systems \textbf{3} (1983), no.~3,
  351--385. \MR{741393}

\bibitem[MU03]{MR2003772}
R.~D. Mauldin and M.~Urba{\'n}ski, \emph{Graph directed {M}arkov systems},
  Cambridge Tracts in Mathematics, vol. 148, Cambridge, 2003.

\bibitem[MW86]{MR860394}
R.~D. Mauldin and S.~C. Williams, \emph{On the {H}ausdorff dimension of some
  graphs}, Trans. Amer. Math. Soc. \textbf{298} (1986), no.~2, 793--803.
  \MR{860394}

\bibitem[Pat97]{MR1479016}
N.~Patzschke, \emph{Self-conformal multifractal measures}, Adv. in Appl. Math.
  \textbf{19} (1997), no.~4, 486--513. \MR{1479016}

\bibitem[Pes97]{pesindimensiontheoryMR1489237}
Y.~B. Pesin, \emph{Dimension theory in dynamical systems}, Chicago Lectures in
  Mathematics, University of Chicago Press, Chicago, IL, 1997, Contemporary
  views and applications. \MR{MR1489237 (99b:58003)}

\bibitem[Roc70]{rockafellar-convexanalysisMR0274683}
R.~T. Rockafellar, \emph{Convex analysis}, Princeton Mathematical Series, No.
  28, Princeton University Press, Princeton, N.J., 1970. \MR{MR0274683 (43
  \#445)}

\bibitem[Sch99]{MR1738952}
J.~Schmeling, \emph{On the completeness of multifractal spectra}, Ergodic
  Theory Dynam. Systems \textbf{19} (1999), no.~6, 1595--1616. \MR{1738952}

\bibitem[SS91]{MR1111613}
T.~Sekiguchi and Y.~Shiota, \emph{A generalization of {H}ata-{Y}amaguti's
  results on the {T}akagi function}, Japan J. Indust. Appl. Math. \textbf{8}
  (1991), no.~2, 203--219. \MR{1111613}

\bibitem[Sum11]{s11random}
H.~Sumi, \emph{Random complex dynamics and semigroups of holomorphic maps},
  Proc. London Math. Soc. (1) (2011), no.~102, 50--112.

\bibitem[Sum13]{S13Coop}
\bysame, \emph{Cooperation principle, stability and bifurcation in random
  complex dynamics}, Adv. Math. \textbf{245} (2013), 137--181.

\bibitem[Wal82]{MR648108}
P.~Walters, \emph{An introduction to ergodic theory}, Graduate Texts in
  Mathematics, vol.~79, Springer-Verlag, New York, 1982. \MR{MR648108
  (84e:28017)}

\end{thebibliography}
\end{document}